\documentclass[11pt,twoside]{amsart}
\usepackage{amssymb,amsthm,amsmath,amstext}
\usepackage{mathrsfs}  
\usepackage{bm}        
\usepackage{mathtools} 
\usepackage{graphicx}
\usepackage[all]{xy}
\usepackage{tikz}
\usetikzlibrary{intersections}

\usepackage[left=1.5in,top=1.49in,right=1.5in,bottom=1.47in]{geometry}

\usepackage{amscd}
\usepackage[latin1]{inputenc}
\usepackage[OT1]{fontenc}
\usepackage{stmaryrd}

\theoremstyle{plain}
\newtheorem*{theorem*}{Theorem}
\newtheorem{theorem}{Theorem}[section]

\newtheorem{proposition}[theorem]{Proposition}
\newtheorem{lemma}[theorem]{Lemma}
\newtheorem{corollary}[theorem]{Corollary}
\newtheorem{notation}[theorem]{Notation}

\newtheorem{construction}[theorem]{Construction}

\theoremstyle{definition}
\newtheorem{definition}[theorem]{Definition}
\newtheorem{example}[theorem]{Example}

\newtheorem{remark}[theorem]{Remark}

\newcommand{\sheaf}[1]{\mathscr{#1}}

\newcommand{\OO}{\sheaf{O}}

\newcommand{\EE}{\sheaf{E}}

\newcommand{\HH}{\sheaf{H}}

\newcommand{\KK}{\sheaf{K}}

\newcommand{\isom}{\cong}

\newcommand{\Z}{\mathbb Z}

\newcommand{\C}{\mathbb C}
\renewcommand{\P}{\mathbb P}

\newcommand{\N}{\mathbb{N}}
\newcommand{\F}{\mathbb{F}}
\newcommand{\G}{\mathbb{G}}
\newcommand{\Gm}{\G_m}

\DeclareMathOperator{\Br}{\mathrm{Br}}

\DeclareMathOperator{\Pic}{\mathrm{Pic}}

\DeclareMathOperator{\Spec}{\mathrm{Spec}}

\DeclareMathOperator{\CH}{\mathrm{CH}}
\DeclareMathOperator{\coker}{\mathrm{coker}}

\newcommand{\ur}{\mathrm{nr}}
\newcommand{\Hur}{H_{\ur}}
\newcommand{\et}{\mathrm{\acute{e}t}}
\newcommand{\Het}{H_{\et}}

\newcommand{\mapto}[1]{\xrightarrow{#1}}

\newcommand{\id}{\mathrm{id}}

\newcommand{\barD}{\overline{D}}
\newcommand{\blowup}{\widetilde{\P}^2}
\newcommand{\tildeC}{\widetilde{C}}
\newcommand{\barC}{\overline{C}}
\newcommand{\tildeLi}{\widetilde{L_i}}
\newcommand{\Lcontact}{L_{c}}
\newcommand{\barB}{\overline{B}}

\usepackage{hyperref}
\hypersetup{pdftitle={Conic bundles with nontrivial unramified Brauer group over threefolds}}
\hypersetup{pdfauthor={Asher Auel, Christian Boehning, Hans-Christian Graf von Bothmer, and Alena Pirutka}}
\hypersetup{colorlinks=true,linkcolor=blue,anchorcolor=blue,citecolor=blue}

\begin{document}

\title[Conic bundles with nontrivial unramified Brauer group]{Conic bundles with nontrivial unramified\\ Brauer group over threefolds}

\author[Auel]{Asher Auel}
\address{Asher Auel, Department of Mathematics\\
Yale University\\
10 Hillhouse Avenue\\
New Haven, CT 06511, USA}
\email{asher.auel@yale.edu}

\author[B\"ohning]{Christian B\"ohning}
\address{Christian B\"ohning, Mathematics Institute, University of Warwick\\
Coventry CV4 7AL, England}
\email{C.Boehning@warwick.ac.uk}

\author[Bothmer]{Hans-Christian Graf von Bothmer}
\address{Hans-Christian Graf von Bothmer, Fachbereich Mathematik der Universit\"at Hamburg\\
Bundesstra\ss e 55\\
20146 Hamburg, Germany}
\email{hans.christian.v.bothmer@uni-hamburg.de}

\author[Pirutka]{Alena Pirutka}
\address{Alena Pirutka, Courant Institute of Mathematical Sciences,
New York University\\ 
251 Mercer Street\\ 
New York, NY 10012, USA}
\email{pirutka@cims.nyu.edu}

\begin{abstract}
We derive a formula for the unramified Brauer group of a general class
of rationally connected fourfolds birational to conic bundles over
smooth threefolds.  We produce new examples of conic bundles over
$\P^3$ where this formula applies and which have nontrivial unramified
Brauer group.  The construction uses the theory of contact surfaces
and, at least implicitly, matrix factorizations and symmetric
arithmetic Cohen--Macaulay sheaves, as well as the geometry of special
arrangements of rational curves in $\P^2$.  We also prove the
existence of universally $\CH_0$-trivial resolutions for the general
class of conic bundle fourfolds we consider.  Using the degeneration
method, we thus produce new families of rationally connected fourfolds
whose very general member is not stably rational.
\end{abstract}

\maketitle


\section{Introduction}\label{sIntroduction}


One of the fundamental problems in the birational classification of
algebraic varieties is to distinguish between varieties that are in
some sense close to $\P^n$---e.g., stably rational, unirational, or
rationally connected---and varieties in the birational equivalence
class of $\P^n$ itself.  Conic bundles over rational varieties are a
natural class to study in this respect, and the literature on them is
prodigious.  For example, conic bundles over rational surfaces were
used in \cite{AM72} to produce varieties that are unirational but not
stably rational (hence a fortiori not rational), and in
\cite{B-CT-S-SwD} to produce stably rational, but non-rational
varieties.  In \cite{CT-O}, the unramified cohomology groups were
introduced to give a more systematic treatment of, and greatly
generalize, the examples in \cite{AM72}. There is also a whole body of
work on conic bundles that are birationally rigid, taking its
departure from the groundbreaking works \cite{Sa80}, \cite{Sa82},
\cite{Is87}; see \cite{Pukh13} for a survey.

Conic bundles are important from a deformation-theoretic perspective
as well, as they usually come in families, making them amenable to the
degeneration method introduced and developed in the seminal articles
\cite{Voi15} and \cite{CT-P16}.  The method relies on the ability to
obstruct the universal triviality of the Chow group of $0$-cycles on a
mildly singular central fiber of such a family.  Then the very general
fiber of the family will be similarly obstructed, and in particular,
will not be stably rational.  The degeneration method has broadened
the range of applicability of previously known obstructions such as
unramified invariants and differential forms in positive
characteristic, and notably, has very recently led to examples of
families of smooth fourfolds with rational and non-rational fibers
\cite{HPT16}.

The present article started from a close analysis of the example in
\cite{HPT16} of a quadric surface fibration over $\P^2$ with
nontrivial unramified Brauer group, defined as divisor of bi-degree
$(2,2)$ in $\P^2\times \P^3$.  While the projection to $\P^2$ gives
the quadric surface fibration structure over $\P^2$, the other
projection gives a conic bundle over $\P^3$. The structural features
of this conic bundle helped us find the statements of the general
results of Section~\ref{sBrauerConic} about the unramified Brauer
group and of Section~\ref{sSingularities} about the singularities of
conic bundles over threefolds.  We also provide new constructions, in
Sections~\ref{sHPT}, \ref{sCatanese}, and~\ref{sIntersectionCurve}, of
conic bundles where these results apply.  One application is the
following (see Theorem~\ref{tMainApplication}).
\begin{theorem*}
A very general conic bundle $Y \to \P^3$ over $\C$, defined by a
homogeneous $3 \times 3$ matrix with entries of degrees
\[
    \begin{pmatrix}
        7 & 4 & 4 \\
        4 & 1 & 1 \\
        4 & 1 & 1 
     \end{pmatrix}
\]
is not stably rational.
\end{theorem*}
Let us describe the contents of the individual Sections in more detail.

\medskip

In Section~\ref{sBrauerConic}, we provide a formula for the unramified
Brauer groups of the total spaces of certain conic bundles over smooth
projective threefolds $B$ with $\Br(B)[2]=0$ and $\Het^3(B, \Z /2)=0$
over an algebraically closed field $k$ of characteristic not $2$.  The
formula (given in Theorem~\ref{tAuelPirutkaAlgebraicVersion}) depends
on the geometry and combinatorics of the components of the
discriminant divisor and their mutual intersections, as well as the
structure of their double covers induced by the lines in the fibers of
the conic bundle.  If the discriminant is irreducible, the unramified
Brauer group of the conic bundle is trivial.  The formula can be
viewed as a higher dimensional analogue of a formula due to
Colliot-Th\'el\`ene (see \cite[Thm.~3.13]{Pi16}) for conic bundles
over surfaces, see also \cite{zagorski}.  Such formulas are naturally
stated in the language of Galois cohomology, algebraic $K$-theory, and
Bloch--Ogus theory, but we go on to reinterpret ours in a geometric
way in Corollary~\ref{cAuelPirutkaMoreGeometric}. This is fundamental
for finding, in Sections \ref{sCatanese} and \ref{sIntersectionCurve},
the geometric examples of conic bundles where the formula applies.

In Section~\ref{sHPT}, we introduce a method to produce fourfold conic
bundles with reducible discriminants via taking double covers branched
in surfaces that are contact to discriminants of simpler conic
bundles. We analyze the example in \cite{HPT16}, of a divisor of
bidegree $(2,2)$ in $\P^2\times \P^3$, as a conic bundle over $\P^3$
from this perspective, yielding an independent proof that this variety
has nontrivial unramified Brauer group.

In Section~\ref{sCatanese}, we introduce another method to construct
fourfold conic bundles over $\P^3$ with reducible discriminants. It is
again based on the theory of contact of surfaces developed largely in
the fundamental paper \cite{Cat81}, as well as on the theory of matrix
factorizations as in \cite{Ei80} and the theory of symmetric
determinantal representations of hypersurfaces \cite{Cat81},
\cite{Beau00}, \cite[Chapter 4]{Dol12}.  While the latter two
theoretical tools are not used logically in our proof, they were very
important in finding the result.

In Section~\ref{sIntersectionCurve}, we complete the construction of
new examples of fourfold conic bundles over $\P^3$ with nontrivial
unramified Brauer group. These are, hence, not stably rational. They
are part of natural families of conic bundles of specific graded-free
types over $\P^3$.

Finally, in Section~\ref{sSingularities}, we analyze the singularities
of the total spaces of a quite general class of conic bundle
fourfolds, proving that they admit universally $\CH_0$-trivial
resolutions. This is aided by a classification of local analytic
normal forms for the singularities that can appear.  The degeneration
method of \cite{Voi15} and \cite{CT-P16} can then be applied to yield
an obstruction to stable rationality of the very general member of
families in which our new examples appear.  In particular, this
provides a simpler proof that the example considered in \cite{HPT16}
admits a universally $\CH_0$-trivial resolution.

As a final note, it may be interesting to remark that we were only
able to construct the examples in Sections~\ref{sCatanese} and
\ref{sIntersectionCurve} by translating virtually every algebraic
concept entering in Theorem~\ref{tAuelPirutkaAlgebraicVersion} into
geometry.  In this respect, hypersurfaces with symmetric rank $1$
arithmetic Cohen--Macaulay sheaves are better than determinants,
contact of surfaces is a more versatile concept than reducibility of
polynomials, and special configurations of rational curves are more
concrete than the analysis of functions becoming squares when
restricted to a curve. On the other hand, the arithmetic
function-field and Galois cohomological point of view is far superior
if one wants to prove an abstract general result such as
Theorem~\ref{tAuelPirutkaAlgebraicVersion}.  The main difficulty is
then constructing examples.  One reason why it is so much more
difficult to find conic bundles over threefolds with prescribed
discriminant, as opposed to over surfaces, is that the theory of
maximal orders in quaternion algebras over threefolds is more
complicated.  Instead of relying on the theory of maximal orders, 
which was utilized in \cite{AM72}, we rely on geometry to construct
our examples.

\medskip

\textbf{Conventions.} The letter $k$ will usually denote an
algebraically closed ground field of characteristic not $2$, unless
explicitly stated otherwise. As usual, the term variety over $k$ means
a separated, integral scheme of finite type over $k$. A conic bundle
is a flat projective surjective morphism of varieties with (geometric)
fibers isomorphic to plane conics and general fiber smooth.

\medskip

\textbf{Acknowledgments.} We would like to thank Fabrizio Catanese,
Jean-Louis Colliot-Th\'el\`ene, and Alexander Kuznetsov for useful
discussions and suggestions during the course of this project.  We
would also like to thank the organizers of the following meetings
where (various subgroups of) the authors started this work: the Simons
Symposium ``Geometry over non-closed fields'' held April 18--22, 2016
in Schlo{\ss} Elmau, Bavaria; the Edge Days workshop ``Birational
geometry and reduction to positive characteristic'' held June 3--5 at
the University of Edinburgh, Scotland; and the CIMI conference ``New
methods in birational geometry'' held June 27--July 1 at the
University of Toulouse, France.  The first author was partially
supported by the NSA grant H98230-16-1-032.

\pagebreak


\section{Brauer group of conic bundles over threefolds} \label{sBrauerConic}


We first recall a few facts from Galois cohomology.

Let $L$ be the function field of an integral variety $Z$ defined over
$k$. At this point we do not even have to assume that $k$ is
algebraically closed, but $k$ should have characteristic different
from two. The first Galois cohomology group $H^1 (L, \Z /2) := H^1
(\mathrm{Gal}(L), \Z /2)$, with constant coefficients $\Z /2$, can be
identified via Kummer theory with the group of square classes
\begin{gather}\label{fGaloisH1}
H^1 (L, \Z /2) \simeq L^{\times} / L^{\times 2}.
\end{gather}
The second Galois cohomology group $H^2 (L, \Z/ 2)$ can be identified
with the $2$-torsion subgroup of the Brauer group of $L$
\begin{gather}\label{fGaloisH2}
H^2 (L, \Z /2) \simeq \Br(L)[2].
\end{gather}
For $a,b \in L^{\times}$, we denote by the symbol $(a,b) \in
\Br(L)[2]$ the Brauer class of the quaternion algebra generated by $x,
y$ with relations $x^2=a$, $y^2=b$, and $xy=-yx$.  This is the same as
the Brauer class associated to the plane conic over $L$ defined by $a
x^2 +b y^2 = z^2$.  It also coincides with the cup product of the
square classes of $a$ and $b$ via the identification
\eqref{fGaloisH1}.

Now suppose $D$ is a prime divisor of $Z$ such that $Z$ is regular in
the generic point of $D$; thus $D$ corresponds to a unique discrete
divisorial valuation $v_D$ of $L$ with residue field $k(D)$. We can
then define two residue maps (homomorphisms) relevant to us in the
sequel
\begin{align}
\label{fResidueMaps}
\begin{split}
\partial^1_D & \colon H^1 (L, \Z /2)  \to H^0 (k(D), \Z /2) = \Z /2 \\ 
\partial^2_D & \colon H^2 (L, \Z /2) \to H^1 (k(D), \Z /2)
\end{split}
\end{align}
in the following manner: if a class in $H^1 (L, \Z/2)$ is represented
by an element $a\in L^{\times}$ according to \eqref{fGaloisH1}, then
$\partial^1_D (a) = v_D(a) \pmod 2$; if a class in $H^2 (L, \Z /2)$
is represented by a symbol $(a,b)$ according to \eqref{fGaloisH2}, then
\begin{gather}\label{fResH2}
\partial^2_{D} (a,b) = (-1)^{v_D(a)v_D(b)} \overline{
a^{v_D(b)}/b^{v_D(a)}} 
\end{gather} 
where $\overline{a^{v_D(b)}/b^{v_D(a)}} \in H^1 (k(D), \Z/2) =
k(D)^{\times}/k(D)^{\times 2}$ is the square class of the unit
$a^{v_D(b)}/b^{v_D(a)} \in L^{\times}$ in the residue field.  In fact
$\partial^2_{D}$ is uniquely determined by the formula
$\partial^2_{D}(\pi,u)=\overline{u}$ for any uniformizer $\pi$ and
unit $u$ in the valuation ring of $v_{D}$.  For $u \in L^\times$, we
sometimes write $u|_D := \overline{u}$ for the residue class.

One also defines the map $\partial^1_D$ in the more general case when
$Z$ is potentially singular at the generic point of $D$, so that the
local ring of $Z$ at the generic point of $D$ is not necessarily a
discrete valuation ring. In that case, we define $\partial^1_D$
following Kato~\cite[p.~151]{Ka86}. If $Z' \to Z$ is the normalization
and $D_1, \dotsc , D_{\mu}$ are the irreducible components lying over
$D$ corresponding to the discrete divisorial valuations of $L$ with
center $D$, then for $a\in L^{\times}$ we define
\begin{gather}\label{fRes1General}
\partial^1_D (a) = \sum_{i=1}^{\mu} [k(D_i):k(D)] v_{D_i} (a) \quad
\pmod 2.
\end{gather}

The unramified cohomology group $H^2_{\mathrm{nr}} (L/k, \Z/2)$, which
depends on the ground field $k$, is the subgroup of $H^2 (L, \Z /2)$
consisting of those elements that are annihilated by all residue maps
$\partial^2_v \colon H^2 (L, \Z/2) \to H^1 (\kappa (v), \Z/2)$ where
$v$ runs over the divisorial valuations of $L$ that are trivial on
$k$. Here $\kappa (v)$ is the residue field of $v$. Clearly, formula
(\ref{fResH2}) makes sense for any divisorial valuation $v$ of $L$, not
only those $v_D$ that have a divisorial center $D$ on $Z$.  The
nontriviality of the unramified cohomology group is an obstruction to
stable rationality of $L$ over $k$.

If $Z$ is smooth and proper over $k$, then there is a natural
isomorphism $\Br(Z)[2]\to H^2_{\mathrm{nr}} (L/k, \Z/2)$, where
$\Br(Z)=\Het^2 (Z, \G_m)$ is the cohomological Brauer group of $Z$,
cf.\ \cite[Prop.~4.2.3(a)]{CT95}.  In general, we refer to
$H^2_{\mathrm{nr}} (L/k, \Z/2)$ as the $2$-torsion in the unramified
Brauer group, and write it as $\mathrm{Br}_{\mathrm{nr}}(L/k)[2]$.

It is of course impossible to check \emph{all} divisorial valuations
in the definition of unramified cohomology just given, so in practice
one needs some complimentary result that narrows this set down to a
set of valuations corresponding to prime divisors on a \emph{fixed}
model of $L$. Such results are implied by so-called ``purity"
\cite{CT95} and we will use a variant of \cite[Thm.~3.82]{CT95}, see
also \cite{Pi16}, Prop. 3.2:

\begin{proposition}\label{pPurityResult}
Let $\OO$ be the local ring of a smooth (scheme-theoretic) point on a
variety over a field $k$ of characteristic not $2$, and let $L$ be the
field of fractions of $\OO$. Let $\gamma \in H^i (L, \Z/2)$ be some
class such that $\partial^i_v (\gamma )=0$ for all valuations
corresponding to height one prime ideals of $\OO$ (hence prime
divisors in $\mathrm{Spec}(\OO )$). Then $\gamma$ is in the image of
the natural map $\Het^i (\mathrm{Spec}( \OO ) , \Z/2) \to H^i (L, \Z
/2)$.
\end{proposition}

The following Corollary is a little more geometric, cf. \cite[Prop.~2.1.8(d)]{CT95}.

\begin{corollary}\label{cPurityResult}
Suppose $Z_\mathrm{sm}$ is a smooth variety over a field $k$ of
characteristic not $2$, and let $L$ be the function field of
$Z_\mathrm{sm}$. Then every element in $H^i (L/k, \Z /2)$ that is
unramified with respect to divisorial valuations corresponding to
prime divisors on $Z_\mathrm{sm}$ is also unramified with respect to
all divisorial valuations that have centers on $Z_\mathrm{sm}$.
\end{corollary}

We will often apply the corollary above to the smooth locus
$Z_\mathrm{sm} := Z\setminus Z_{\mathrm{sing}}$ of a proper variety
$Z$ over $k$, where $Z_{\mathrm{sing}}$ is its singular locus.

\medskip

Let $K$ be an arbitrary field (possibly of characteristic 2) and let
$C$ be a smooth projective curve of genus zero over $K$.  The
anticanonical class on $C$ defines an embedding $C \to \P^2_K$ as a
smooth plane conic; we call $C$ a smooth conic over $K$.  As remarked
earlier, a smooth conic $C$ determines a Brauer class $\alpha \in
\Br(k)[2]$.  We say that $C$ is nonsplit if $C(K) = \emptyset$,
equivalently, $\alpha$ is nontrivial.  As before, we set $\Br(C) :=
\Het^2(C,\Gm)$.  Since $\Br(K) = H^2(K,\Gm) = \Het^2(\Spec K, \Gm)$
for any field $K$, we have a pullback map $\Br(K) \mapto{\iota}
\Br(C)$. We will need the following.

\begin{lemma}\label{coKerConic} 
Let $C$ be a smooth nonsplit conic over an arbitrary field $K$.  Then
the pullback map induces an exact sequence
\begin{gather}\label{sBrConic}
0\to \Z/2\to \Br(K) \mapto{\iota} \Br(C)\to 0
\end{gather}
where the kernel is generated by the Brauer class $\alpha \in \Br(K)[2]$ determined by $C$.

Assuming that $K$ has characteristic not 2 and that $-1$ is a square,
then \eqref{sBrConic} restricts to an exact sequence
\begin{gather}\label{sBrConic2}
0\to \Z/2\to \Br(K)[2] \to \Br(C)[2] \to \Z/2 \to 0.
\end{gather}
and any class not in the image of $\Br(K)[2] \to \Br(C)[2]$ is
contained in the image of $\Br(K)[4] \to \Br(C)[4]$.
\end{lemma}
\begin{proof}
The proof of \eqref{sBrConic} is well known, but we summarize it here
for convenience, cf.\ \cite[Prop.~1.5]{CT-O}.  The identification of
the kernel of $\iota$ is due to Witt~\cite{Wit35}, and follows from
the fact that $C$ is a Severi--Brauer variety associated to the Brauer
class $\alpha$.  The proof of the surjectivity of $\iota$ follows an
argument with the Hochschild--Serre spectral sequence going back to
the work of Lichtenbaum~\cite{Lic69}, Iskovskikh, and Manin.  We
recall this argument here for convenience.  Let $K^s$ be a separable
closure of $K$ and $\Gamma$ the Galois group of $K^s/K$.  The exact
sequence of low degree terms of the Hochschild--Serre spectral
sequence and Hilbert's theorem 90 gives
\begin{gather*}
0 \to \Pic(C) \to \Pic(C_{K^s})^\Gamma \to \Br(K) \to \ker \bigl(\Br(C) \to
\Br(C_{K^s})\bigr) \to H^1(\Gamma, \Pic(C_{K^s}))
\end{gather*}
Since $C$ is a smooth conic, it has a separable splitting field by a
result of Noether, hence $C_{K^s} \isom \P^1_{K^s}$.  For the
vanishing of $\Br(\P^1_{K^s})$, one can appeal to (a generalization
of) Tsen's theorem on the vanishing of the Brauer group of the
function field of a curve over a separably closed field.
We also use the fact that $\Pic(\P^1_{K^s}) = \Z$ has trivial Galois
action and $H^1(\Gamma,\Z)=0$, while $\Pic(C)$ is generated by
$\omega_C^\vee$, which has degree 2, when $C$ is a nonsplit conic.  Hence the above sequence of
low-degree terms collapses to the desired exact sequence.

As for the second part, the fact that any element of $\Br(C)[2]$ is in
the image of $\Br(K)[4] \to \Br(C)[4]$ follows immediately from
\eqref{sBrConic}, since the kernel has order 2.  For the calculation
of the cokernel of $\iota$, the short exact sequence of group schemes
$1 \to \mu_2 \to \mu_4 \to \mu_2 \to 1$ (assuming that $K$ has characteristic not 2) induces a long exact sequence
in Galois cohomology 
\begin{gather*}
\dotsm \to H^1(K,\mu_2) \to H^2(K,\mu_2) \to H^2(K,\mu_4) \to H^2(K,\mu_2)
\to H^2(K,\mu_2) \to \dotsm
\end{gather*}
where the boundary maps are given by cup product with the class $(-1)
\in H^1(K,\Z/2)$, cf.\ \cite[Lemmas~1,2]{Kah89}. Hence all boundary
maps are zero if $-1$ is a square in $K$.  Since $K$ has
characteristic not 2, we have $\Br(K)[n] = H^2(K,\mu_{n})$ for $n$ a
power of 2.  We then have the following commutative diagram with exact
rows
\begin{gather*}
\xymatrix{
0 \ar[r] & \Br(K)[2] \ar[r]\ar[d] & \Br(K)[4] \ar[r]\ar[d] & \Br(K)[2]
\ar[r]\ar[d] & 0 \\
0 \ar[r] & \Br(C)[2] \ar[r]       & \Br(C)[4] \ar[r]       & \Br(C)[2] &  
}
\end{gather*}
and the snake lemma yields that
$$
\coker\bigl(\Br(K)[2] \to \Br(C)[2]\bigr) \isom \ker\bigl(\Br(K)[2] \to \Br(C)[2]\bigr)
\isom \Z/2
$$ 
as desired, cf.\ \cite[\S7]{kahn_rost_sujatha}.  We use the fact that
$\Br(K)[4] \to \Br(C)[4]$ maps onto $\Br(C)[2]$ to see that the map
$$
\coker\bigl(\Br(K)[2] \to \Br(C)[2]\bigr) \to \coker\bigl(\Br(K)[4] \to \Br(C)[4]\bigr)
$$
is zero, even though $\coker\bigl(\Br(K)[4] \to \Br(C)[4]\bigr)$ might
itself be nonzero.
\end{proof}

\begin{definition}\label{dDiscriminant}
Let $\pi \colon Y \to B$ be a conic bundle over a smooth projective
threefold $B$ over an algebraically closed ground field $k$ of
characteristic not $2$.  Let $S$ be the locus of points in $B$ such
that for any (closed) point $s\in S$, the fiber $Y_s$ is singular. We
call $S$ the \emph{discriminant locus} of $\pi$. Let $S_1, \dots ,
S_n$ be its irreducible components; each $S_i$ is then an irreducible
surface.

We call the discriminant locus $S$ \emph{good} if for each $i$, the
fiber $Y_s$ for general $s\in S_i$ consists of two distinct lines, and
the natural double covers $\tilde{S}_i \to S_i$ determined by $\pi$ in
that case are irreducible.
\end{definition}

\begin{remark}\label{rGood}
Keeping the notation of the previous definition, if $S$ is good and
$\alpha\in H^2(K, \Z /2)$ is the Brauer class corresponding to the
generic fiber of $\pi$, then the surfaces $S_i$ are precisely those
surfaces $\Sigma \subset B$ such that $\partial^2_{\Sigma} (\alpha )
\neq 0$. If we drop the assumption that the cover $\tilde{S}_i \to
S_i$ be irreducible, then we could get a trivial class in $H^1
(k(S_i), \Z /2)= k(S_i)^{\times}/k(S_i)^{\times 2}$.
\end{remark}

We can now go back to our geometric situation and state an algebraic
version of the theorem that computes $H^2_{\mathrm{nr}}(k(Z)/k, \Z
/2)$ for us in many cases.

\begin{theorem}\label{tAuelPirutkaAlgebraicVersion}
Let $k$ be an algebraically closed field of characteristic not $2$ and
let $\pi \colon Y \to B$ be a conic bundle over a smooth projective
threefold $B$ over $k$. Let $\alpha \in \mathrm{Br}(K)[2]$ be the
Brauer class in $K =k (B)$ corresponding to the generic fiber of
$\pi$.  Assume that the discriminant locus of $\pi$ is good with
components $S_1, \dots , S_n$.  We will also assume the following:
\begin{enumerate}
\item\label{cThreefold} The vanishing $\Br(B)[2]=0$ and
$\Het^3(B, \Z /2) =0$ holds.
\item\label{cCombinatoricsIntersection1} Through any irreducible curve
in $B$, there pass at most two surfaces from the set $S_1, \dots ,
S_n$.
\item\label{cCombinatoricsIntersection2} Through any point of $B$,
there pass at most three surfaces from the set $S_1, \dots, S_n$.
\item\label{cLocalStructure} For all $i\neq j$, $S_i$ and $S_j$ are
factorial at every point of $S_i\cap S_j$.
\end{enumerate}
Put
\[
\gamma_i = \partial^2_{S_i} (\alpha ) \in H^1 (k (S_i) , \Z /2).
\]
Define a subgroup $\Gamma$ of the group $\bigoplus_{i=1}^n H^1(k(S_i), \Z/2)$ by
\[
\Gamma=\bigoplus_{i=1}^n \langle \gamma_i \rangle .
\]
Thus $\Gamma \simeq (\Z /2)^n$. We will write elements of $\Gamma$ as
$(x_1, \dots , x_n)$ with $x_i \in \{ 0,1\}$.

Let $H\subset \Gamma$ consist of those elements $(x_1, \dots ,
x_n)\in (\Z /2)^n$ such that $x_i=x_j$ for $i\neq j$ whenever there
exists an irreducible component $C$ of $S_i\cap S_j$ such that either
\begin{itemize}
\item[\it i)] $\partial^1_C(\gamma_i ) = \partial^1_C(\gamma_j) =1$, or

\item[\it ii)] $\partial^1_C (\gamma_i) = \partial^1_C (\gamma_j) =0$ and
$\gamma_i|_C$ and $\gamma_j|_C$ 
are not both zero in $H^1 (k (C), \Z/2)$.
\end{itemize} 
Then the 2-torsion of the unramified Brauer group $H^2_{\mathrm{nr}}
(k (Y)/k , \Z/2)$ of $Y$ contains the subquotient $H /\langle (1,
\dots , 1)\rangle$ by the ``diagonal subgroup" $\langle (1, \dots ,
1)\rangle$ of $\Gamma$, and is equal to it under the following
additional geometric assumption
\begin{itemize}
\item[\it iii)] If $\partial^1_C (\gamma_i) = \partial^1_C (\gamma_j)
=0$ in an irreducible component of the intersection $S_i \cap S_j$,
then $S_i$ and $S_j$ intersect generically transversally along $C$ and
the rank of the conics in the fibers of $Y$ is generically $2$ over
$C$.
\end{itemize}
\end{theorem}

Later, we will reformulate various portions of
Theorem~\ref{tAuelPirutkaAlgebraicVersion} more geometrically. Before
embarking on the proof, a few explanatory remarks are in order.

\begin{remark}\label{rExtraAssumption}
We do not know if the assumption iii) is necessary or redundant, i.e.,
whether we have equality $H^2_{\mathrm{nr}} (k (Y)/k , \Z/2)=H
/\langle (1, \dots , 1)\rangle$ without it. It is conceivable that in
any case there is a conic bundle $Y' \to B$, birational to $Y$ over
$B$, such that iii) is satisfied. However, for us iii) serves as a
harmless simplifying assumption.
\end{remark}

\begin{remark}\label{rConditions}
Conditions \ref{cCombinatoricsIntersection1}) and
\ref{cCombinatoricsIntersection2}) are obviously simplifying
assumptions on the intersection graph of the $S_1, \dots , S_n$. They
could be replaced by different ones, but this would make the
description of the unramified Brauer group $H^2_{\mathrm{nr}} (k (Y)/k
, \Z/2)$ messier. On the other hand, condition \ref{cLocalStructure})
is a hypothesis on the local algebraic structure, and something of
that sort is probably indispensable in any version of Theorem
\ref{tAuelPirutkaAlgebraicVersion}.  Condition \ref{cThreefold}) is
needed to glue certain Galois $H^1$-classes into Brauer classes on $B$
as we will see below.
\end{remark}


\begin{proof}[Proof of Theorem \ref{tAuelPirutkaAlgebraicVersion}]
It is a bit lengthy and we divide it into steps to make the logic clearer.

\medskip

\textbf{Step 1.} \emph{Inducing all potentially unramified Brauer
classes in $H^2_{\mathrm{nr}} (k (Y)/k, \Z /2)$ from Brauer classes on
$B$ that are glued from a compatible set of
$\gamma_i=\partial^2_{S_i}(\alpha )$.}  The first question is how we
can describe a totality of classes in $H^2 (k(Y), \Z /2)$ that are the
only candidates to yield unramified classes in $H^2_{\mathrm{nr}} (k
(Y)/k , \Z /2)$. This is done via the following commutative diagram:

{\tiny
\begin{gather}\label{dBasic}
\xymatrix@R=14pt{
& & 0 & &\\
& & \Z/2\ar[u] & &\\
0 \ar[r] & H^2_{\mathrm{nr}} (k (Y)/Y , \Z /2) \ar[r] & H^2_{\mathrm{nr}} (k (Y)/K, \Z /2) \ar[u]\ar[r]^{\oplus \partial^2_{T} \quad\quad} & \bigoplus_{T \in Y^{(1)}_B} H^1 (k (T), \Z /2) \\
            & \mathrm{Br}_{\mathrm{nr}}(K)[2] =0 \ar[r] & H^2 (K, \Z /2) = \mathrm{Br} (K) [2] \ar[r]^{\oplus \partial^2_S}  \ar[u]^(.45){\iota }               & \bigoplus_{S \in B^{(1)}} H^1 (k (S), \Z /2)\ar[u]^(.38){\tau} \ar[r]^{\oplus (\oplus \partial^1_C)} & \bigoplus_{C \in B^{(2)}} H^0 ( k (C), \Z /2) \\
           &                                                                           & \langle \alpha \rangle   \ar[u]                                                                                                    & \KK  \ar[u]                                &        \\
           &                                                                           &                       0 \ar[u]                                                                                                                    &  0 \ar[u]                                                                                                                                         & 
}
\end{gather}
}

We will start by explaining the new pieces of notation:
$H^2_{\mathrm{nr}} (k (Y)/Y , \Z /2) $ denotes all those classes in
$H^2 (k(Y), \Z/2)$ which are unramified with respect to divisorial
valuations corresponding to prime divisors (threefolds) on $Y$. Note
that the singular locus of $Y$ has codimension $\ge 2$ by our
assumptions. By Corollary \ref{cPurityResult}, we can also
characterize $H^2_{\mathrm{nr}} (k (Y)/Y , \Z /2) $ as all those
classes in $H^2 (k(Y), \Z/2)$ that are unramified with respect to
divisorial valuations which have centers on $Y$ which are not
contained in $Y_{\mathrm{sing}}$. Moreover, $H^2_{\mathrm{nr}} (k
(Y)/K, \Z /2)$ is the subset of those classes in $H^2 (k(Y), \Z/2)$
which are unramified with respect to divisorial valuations that are
trivial on $K$, hence correspond to prime divisors of $Y$ dominating
the base $B$ (since $\pi$ is of relative dimension $1$).

In the upper row, $T$ runs over all irreducible threefolds, i.e.,
prime divisors, in $Y$ that do not dominate the base $B$, hence map to
some surface in $B$. We call this set of irreducible threefolds
$Y^{(1)}_B$. Then the upper row is exact by the very definitions.

In the lower row, $S$ runs over the set of all irreducible surfaces
$B^{(1)}$ in $B$ and $C$ over the set of all irreducible curves
$B^{(2)}$ in $B$.  Thus this row coincides with the usual Bloch--Ogus
complex for degree 2 \'etale cohomology associated to $B$.  The $i$th
cohomology group of this complex is computed by the Zariski cohomology
$H^i(B,\HH^2)$ of the \'etale cohomology sheaf $\HH^i$, which is the
sheafification of the Zariski presheaf $U \mapsto \Het^2(U,\Z/2\Z)$,
see \cite[Thm.~6.1]{bloch_ogus}.  In particular, the lower row is
exact in the first two places because $H^0(B,\HH^2) = \Br(B)[2] =0$
and $H^1(B,\HH^2) \subset \Het^3(B,\Z/2)=0$ by hypothesis, where the
later inclusion arises from the sequence of low terms associated to
the Bloch--Ogus spectral sequence $H^i(B,\HH^j) \Rightarrow
\Het^{i+j}(B,\Z/2)$, cf.\ \cite[\S1.1]{Kah95}.

Now let us discuss the vertical arrows.  The left vertical column is
Lemma~\ref{coKerConic}.  The map $\tau$, defined by pullback under the
field extensions $k(T) \supset k(S)$, coincides with the induced
$k(S)^\times/k(S)^{\times 2} \to k(T)^\times/k(T)^{\times 2}$.  If the
generic fiber of $T\to S$ is geometrically integral, then $k(S)$ is
algebraically closed inside $k(T)$, hence this induced map is
injective.  This is the case if $S$ is not contained in the
discriminant locus, since then the generic fiber of $T \to S$ is a
smooth conic.  If $S=S_i$ is a component of the discriminant locus,
then the generic fiber of $T_i \to S_i$ is geometrically the union of
two lines; Stein factorization displays this generic fiber as a line
over the quadratic extension $F/k(S_i)$ defined by the residue class
$\gamma_i \in H^1(k(S_i),\Z/2)$.  In this case, the
restriction-corestriction exact sequence in Galois cohomology implies
that the kernel of the natural map $H^1(k(S_i),\Z/2) \to H^1(F,\Z/2)$
is generated by $\gamma_i$ (and also the natural map $H^1(F,\Z/2) \to
H^1(F(t),\Z/2)$ is injective).  We conclude that the kernel of $\tau$  is
\begin{gather}\label{fKernel}
\KK \simeq \langle \gamma_1 \rangle \oplus \dots \oplus \langle \gamma_n \rangle = \Gamma.
\end{gather}

We argue that even though $\iota$ is not surjective, the subgroup
$\Hur^2(k (Y)/Y , \Z /2) \subset \Hur^2(k (Y)/K , \Z /2)$ is in the
image of $\iota$.  By Lemma~\ref{coKerConic}, any element
\[
\zeta \in \Hur^2(k(Y)/K,\Z/2)\]
not in the image of $\iota$ lifts to some $\xi \in H^2(K,\Z/4)$ of
order 4.  Then at least one residue $\partial^2_S(\xi) \in
H^1(k(S),\Z/4)$ must have order 4, since the map $\oplus\partial^2_S$
is injective (i.e., we consider the lower row of diagram
(\ref{dBasic}) with $\Z /4$ coefficients now).  Since $\KK$ is an
elementary abelian 2-group and also equals the kernel of the map
$\tau$ for $\Z/4$ coefficients
\[
\tau \colon \bigoplus_{S\in B^{(1)}} H^1 (k(S), \Z /4) \to \bigoplus_{T\in Y^{(1)}_B} H^1 (k(T), \Z/4), 
\]
$\tau(\partial^2_S(\xi)) \in H^1(k(T),\Z/4)$ cannot be trivial.  Since
the diagram commutes, we see that $\partial^2_T(\zeta)$ is nontrivial,
hence $\zeta$ cannot lie in $\Hur^2(k(Y)/Y,\Z/2)$.  This same diagram
chase for the diagram (\ref{dBasic}) yields that the group
$H^2_{\mathrm{nr}} (k (Y)/Y , \Z /2) $ can be described as the
quotient by $\langle (1,\dots , 1 ) \rangle$ of the subgroup
$H'\subset (\Z /2)^n$ defined only using condition {\it i)} of the
definition of $H$ in the statement of
Theorem~\ref{tAuelPirutkaAlgebraicVersion}.  Note also that we use
assumption \ref{cCombinatoricsIntersection1}) (namely, each $C$
determines a unique pair $S_i, S_j$ such that $C$ is a component of
$S_i\cap S_j$) to ensure that elements in $H'$ make up the kernel of
$\oplus (\oplus \partial^1_C)$ in diagram (\ref{dBasic}).

\medskip

\textbf{Step 2.} \emph{Figuring out which classes in $H'$ give classes in $H^2_{\mathrm{nr}} (k(Y)/k, \Z /2)$ by checking whether they are unramified with respect to all divisorial valuations $\nu$ of $k(Y)$: a case-by-case analysis depending on the dimension and location of the center of $\nu$ on $B$.}

We pick a class $\beta \in H^2 (K, \Z /2)$ corresponding to an element in $H'$, and denote by $\beta'$ the image of $\beta$ in $H^2 (k(Y), \Z /2)$. We want to show that $\beta'$ is unramified on $Y$ if and only if $\beta$ is in $H$. We first prove the 
if part by a case-by-case analysis, and the only if part in Step 3 below. 

\textbf{Step 2. a)} \emph{The center of $\nu$ on $B$ is not contained in the intersection of two or more of the discriminant components.}
Denote by $\OO$ the local ring of the center $Z$ of $\nu$ on $B$. Then $\beta -\alpha$ is in the image of $\Het^2(\OO, \Z /2)$ by Proposition \ref{pPurityResult}. But $\iota (\beta -\alpha ) = \iota (\beta )$, so this class is also unramified with respect to $\nu$ in this case.

\medskip

\textbf{Step 2. b)} \emph{The center $\nu$ on $B$ is a curve $C$ that is an irreducible component of $S_i\cap S_j$.}

Let $\OO$ be the local ring of $C$ in $B$. If $\beta$ has $x_i=x_j=1$,
then again $\beta -\alpha$ is in the image of $\Het^2(\OO, \Z /2)$ by
Proposition \ref{pPurityResult}, and we conclude as before. So we can
assume $x_i=1, x_j=0$ and then also $\partial^1_C(\gamma_i ) =0$. This
condition means that a function representing $\gamma_i
= \partial^2_{S_i}(\beta )\in H^1 (k(S_i), \Z /2)
=k(S_i)^{\times}/k(S_i)^{\times 2}$ has a zero or pole of even order
along $C$. Moreover, $\gamma_j$ can be represented by $1$ in
$k(S_j)^{\times}$. Passing to the inverse of the function representing
$\gamma_i$ if necessary (multiplying by squares does not change its
class in $H^1 (k(S_i), \Z /2)$), we can assume that it is contained in
the local ring $\OO_{S_i,C}$ of $C$ in $S_i$. Call this function
$f_{\gamma_i}$. Choose a local equation $t$ for $C$ in $\OO_{S_i,C}$. Note that $S_i$ is factorial along $C$, so $C$ is a Cartier
divisor on $S_i$.

Then $f_{\gamma_i}/ (t^{v_C(f_{\gamma_i})})$ is a unit in
$\OO_{S_i,C}$, hence any preimage in $\OO$ will be a unit. Call this
preimage $u_{\gamma_i}$. For $u_{\gamma_j}$ we could take $1$. Now
viewing $u_{\gamma_i}$ as a rational function in $K$, the function
field of $B$, and choosing a local equation $\pi_{S_i}$ for $S_i$ in
$\OO$ (also viewed as a function in $K$) we can form the symbol
$(u_{\gamma_i}, \pi_{S_i}) \in H^2 (K, \Z /2)$. Using formula
(\ref{fResH2}), we conclude that
\[
\partial^2_{S_i} (\beta) = \gamma_i = \partial^2_{S_i} (u_{\gamma_i}, \pi_{S_i})
\]
by construction of $u_{\gamma_i}$. Moreover, $\beta - (u_{\gamma_i},
\pi_{S_i})$ is then in the image of $\Het^2(\OO , \Z /2)$
using Proposition \ref{pPurityResult} again. Here we are using that we
have lifted $f_{\gamma_i}$ to a unit $u_{\gamma_i}$ to ensure that
$\partial^2_{S} (u_{\gamma_i}, \pi_{S_i}) =0$ for every other surface
$S$ different from $S_i$ through $C$.
Hence 
\[
\partial^2_{\nu} (\iota (\beta - (u_{\gamma_i}, \pi_{S_i})))=0,  
\]
so we will have shown that $\partial^2_{\nu} (\iota (\beta )) = \partial^2_{\nu}(\beta') =0$ once we know $\partial^2_{\nu}(\iota (u_{\gamma_i}, \pi_{S_i}))=0$. By formula (\ref{fResH2}) we have (up to a sign)

\begin{gather}\label{fResidue1}
\partial^2_{\nu}(\iota (u_{\gamma_i}, \pi_{S_i})) = \overline{u_{\gamma_i}^{\nu (\pi_{S_i})}/\pi_{S_i}^{\nu (u_{\gamma_i})}   } = \overline{u_{\gamma_i}^{\nu (\pi_{S_i})}} \in H^1 (\kappa (\nu ), \Z /2)
\end{gather}
where the second equality follows because $u_{\gamma_i}$ is a unit
along $C$; note that here we are viewing all rational functions in $K$
as functions in $k(Y)$ via the natural extension $K \subset k(Y)$.

On the other hand (up to a sign)
\begin{gather}\label{fResidue2}
\partial_C^2 (\gamma_i,  u^i_C) = \overline{ \gamma_i^{\nu_C (u^i_C)}
/ (u^i_C)^{\nu_C (\gamma_i)}   } = \overline{ \gamma_i^{\nu_C (u^i_C)}   } = \overline{u_{\gamma_i}|_{S_i}}  \in H^1 (k(C), \Z /2)
\end{gather}
where the second equality follows because $\partial^1_C (\gamma_i) =0$
and the third equality because $f_{\gamma_i}$ and the function
$u_{\gamma_i}|_{S_i}$ on $S_i$ differ by a square, by construction.

But since the term in (\ref{fResidue2}) is zero by assumption, so is the term in formula (\ref{fResidue1}).

\medskip

\textbf{Step 2. c)} \emph{The center of $\nu$ is a point $p\in C$ as in Step 2. b), and $S_i, S_j$ are the only surfaces among the $S_1, \dots , S_n$ passing through $p$.}

Let $\OO$ denote the local ring of $p$ in $B$. If $x_i=x_j=1$ we
conclude as above by looking at $\beta - \alpha$. So assume $x_i=1,
x_j=0$. Then $\partial^1_{C}(\gamma_i)=0$. Note that we can find a
local equation $t$ for $C$ in $\OO_{S_i, p}$ since $C$ is Cartier by
the hypothesis that $S_i$ is factorial along $C$. Pick a function
$f_{\gamma_i}\in k(S_i)$ representing $\gamma_i$.  Moreover, for any
other irreducible curve $C'$ passing through $p$, either in $S_i \cap
S_j$ or lying entirely on $S_i$ or $S_j$, we will have
$\partial^1_{C'}(\gamma_i) =0$, too. Let $C_1, \dots , C_N$ be all
irreducible curves through $p$ along which $f_{\gamma_i}$ has a zero
or pole, and pick a local equation $t_{\iota}$ in $\OO_{S_i, p}$ for
every $C_{\iota}$. The rational function $f_{\gamma_i}/ \{
t_1^{v_{C_1}(f_{\gamma_i})} \dotsm
t_N^{v_{C_N}(f_{\gamma_i})} \}$ on $S_i$ does not vanish or have a
pole on any curve on $S_i$ that passes through $p$. Hence, since $S$
is assumed to be factorial, in particular, normal in $p$, this
function is a unit locally around $p$, and can be lifted to a unit in
$\OO$. We call this $u_{\gamma_i}$ again. Repeating the rest of the
proof in Step 2 b) verbatim, with $k(C)$ replaced by $k(P)$, and
using that every element in $k(P)$ is a square since $k$ is
algebraically closed, we see that $\partial_{\nu} (\beta ') = 0$ here
as well.

\medskip

\textbf{Step 2. d)} \emph{The center of $\nu$ is a point $p$ that lies on exactly three surfaces $S_i, S_j, S_k$.}

Then $p \in S_i \cap S_j \cap S_k$. If we have $x_i=x_j=x_k=1$, we can
again pass to $\beta - \alpha$ and argue as above, so we can assume
$x_i=1, x_j=x_k=0$, or $x_i=0, x_j=x_k=1$. Moreover, without loss of
generality, we can assume $\beta$ is of type $x_i=1, x_j=x_k=0$ since
if it is of type $x_i=0, x_j=x_k=1$, $\beta -\alpha$ will be of type
$x_i=1, x_j=x_k=0$, and $\partial_{\nu}(\iota (\beta-\alpha ))
= \partial_{\nu}(\iota (\beta))$.  Let $\OO$ be the local ring of $p$
in $B$ again. Since every curve $C$ on $S_i$ passing through $p$,
either on $S_i\cap S_j$ or $S_i\cap S_k$, or only on $S_i$, is Cartier
on the surface $S_i$, we can find a unit $u_{\gamma_i}$ in $\OO$ that,
when restricted to $S_i$, has the same class as $\gamma_i$ in
$H^1(k(S_i), \Z /2)$. We just repeat the argument in Step 2. b). The
rest of the argument is then verbatim as in Step 2 b) (or Step 2 c))
with $k(C)$ again replaced by $k(P)$.

\medskip

\textbf{Step 3.} \emph{Proving that a class $\beta$ in $H'$ yields an
unramified class $\beta'$ on $Y$ only if $\beta \in H$.}

We have to prove that if $\beta$ has $x_i=1$ and $x_j=0$, so that
$\partial^1_C (\gamma_i) = \partial^1_C (\gamma_j)=0$ for every
irreducible component $C$ of $S_i\cap S_j$, and if $\gamma_i|_C$ and
$\gamma_j|_C$ are nonzero in $H^1 (k (C), \Z/2)$, then $\beta'$ is
ramified with respect to some divisorial valuation $\nu$ of $k(Y)$.

We now make use of assumption iii). Because of this, a local calculation, done later in Proposition~\ref{pNormalForms}, 
shows the following: there is a unique irreducible curve $C'$ in which $Y$ is singular and which dominates $C$ in this case. Also, the map $C' \to C$ is generically one-to-one. 
Moreover, blowing up $Y$ in $C'$ yields an exceptional divisor $E$ that is generically a $\P^1\times \P^1$ bundle over $C'$, 
hence birational to $\P^1\times \P^1 \times C'$. Let $\nu=\nu_E$ be the associated valuation. 
Looking back at the computations in Step 2 above, and keeping the notation there, we see from formula (\ref{fResidue1}) and the fact that 
$\nu_E (\pi_{S_i})=1$ (again a local calculation) that $\partial^2_{\nu } (\beta' )$ is equal to $\bar{u}_{\gamma_i}$, viewed as an element of 
$H^1 (k(E), \Z/2)$. Hence, this is nothing but the image, under the natural map $H^1 (k(C), \Z /2) \to H^1 (k(E), \Z /2)$, of $\bar{u}_{\gamma_i}$, 
viewed as an element of $H^1 (k(C), \Z /2)$. But a nonsquare in a field cannot become a square in a purely transcendental extension of that field, 
hence $\partial^2_{\nu} (\beta' ) \neq 0$ in this case. 
\end{proof}

We can reformulate parts of Theorem \ref{tAuelPirutkaAlgebraicVersion}
to obtain the following geometric Corollary that gives sufficient
conditions for a conic bundle $\pi \colon Y \to B$ to have nontrivial
$H^2_{\mathrm{nr}} (k (Y)/k, \Z /2)$.

\begin{corollary}\label{cAuelPirutkaMoreGeometric}
Let $k$ be again some algebraically closed ground field of
characteristic not equal to $2$, $\pi \colon Y \to B$ a conic bundle
over a smooth projective threefold $B$ with
$\Br(B)[2]=\Het^3(B,\Z/2)=0$.

Suppose that the discriminant locus $S = \bigcup_{i=1}^n S_i$ of $\pi$
is good and $n\ge 2$ and suppose that assumptions
\ref{cCombinatoricsIntersection1}),
\ref{cCombinatoricsIntersection2}), \ref{cLocalStructure}) in Theorem
\ref{tAuelPirutkaAlgebraicVersion} are satisfied.

Suppose that for all $i\neq j$ and every irreducible component $C$ of
$S_i\cap S_j$, the fibers of $\pi$ over a general point of $C$ are
still two distinct lines, and that the corresponding double cover
$\tilde{C} \to C$ (inside $\tilde{S}_i$ or $\tilde{S}_j$) is
reducible.

Then the unramified Brauer group of $Y$ is nontrivial. 
\end{corollary}

\begin{proof}
The fact that the fibers of $\pi$ over a general point of $C$ are
still two distinct lines means $\partial^1_C (\gamma_i) = \partial^1_C
(\gamma_j) =0$. The condition that $\tilde{C}$ is reducible means that
$\partial_C^2 (\gamma_i, u)$ and $\partial^2_C (\gamma_j, u)$ are
zero.
\end{proof}


\section{Reducibility of the discriminant: 1st method}\label{sHPT}


Subsequently, we will usually restrict our attention to conic bundles
of \emph{graded-free type} over $\P^3$, informally, those defined by a
graded symmetric $3 \times 3$ matrix. We now make this precise.

\begin{definition}\label{dGradedFree}
Fix a triple of non-negative integers
\[
(d_1, d_2, d_3 ) \in \N^3  \quad \mathrm{such}\; \mathrm{that}\; d_i \equiv d_j \, \mathrm{(mod}\, 2 \mathrm{)}\: \forall \, i, j.
\]
Consider a symmetric matrix of homogeneous polynomials on $\P^3$
\[
A = \begin{pmatrix}
a_{11} & a_{12} & a_{13} \\
a_{21} & a_{22} & a_{23} \\
a_{31} & a_{32} & a_{33}
\end{pmatrix}
\]
where 
\begin{gather}\label{fGradingConditions}
a_{ij} =a_{ji}, \quad \deg (a_{ii}) = d_i, \quad \deg (a_{ij}) = \frac{d_i+d_j}{2} .
\end{gather}

Put

\begin{gather}\label{fTwists}
d= \sum_i d_i, \quad r_i = \frac{d-d_i}{2} , \quad s_i = \frac{d+d_i}{2} \\
\EE = \OO (r_1) \oplus \OO (r_2) \oplus \OO (r_3) .
\end{gather}

Then $A$ determines a symmetric map between graded free bundles
\[
A\colon \EE(-d) =\OO (-s_1) \oplus \OO (-s_2) \oplus \OO (-s_3) \to  \EE^{\vee} = \OO (-r_1) \oplus \OO (-r_2) \oplus \OO (-r_3) 
\]
hence a line bundle valued map  
\[
\mathrm{Sym}^2 \EE  \to \OO (d)
\]
determining a conic bundle $Y \subset \P (\EE )\to \P^3$ if the entries of $A$ do not vanish simultaneously in any point of $\P^3$. Such a conic bundle will be called of \emph{graded free type}. 
\end{definition}

\begin{example}\label{eCubicFourfold}
If $Y \subset \P^5$ is a cubic hypersurface containing a line $\ell
\subset \P^5$.  The projection $\P^5 \dashrightarrow \P^3$ from $\ell$
is resolved by the blow up $\widetilde{\P}^5$ of $\P^5$ along $\ell$.
The resulting morphism $\widetilde{\P}^5 \to \P^3$ has the structure
of a projective bundle $\P(\EE)$, where $\EE =
\OO(1)\oplus\OO(2)\oplus(2)$.  Restricting this morphism to the blow
up $\widetilde{Y} \subset \widetilde{\P}^5$ of $Y$ along $\ell$, then
$\widetilde{Y} \to \P^3$ is a conic bundle of graded free type
$(3,1,1)$, cf.\ \cite{Tog40}.
\end{example}

We now derive a result saying that certain discriminant surfaces $F$
of conic bundles of graded-free type over $\P^3$ split if pulled back
via a suitable double cover.

\begin{definition}\label{dNode}
A point $p$ on a surface $F$ in $\P^3$ is called a node if \[
\widehat{\OO}_{F, p} \simeq k\llbracket x,y,z\rrbracket /(xy-z^2).\]
\end{definition}

\begin{proposition}\label{pSplitting}
Let $F$ be a surface in $\P^3$ with at most nodes as
singularities. Suppose that for a desingularization $\tilde{F}$ of
$F$, $H^1_{\acute{e}t} (\tilde{F}, \Z/2) =0$, or equivalently,
$H^1_{\mathrm{nr}} (k (F)/k, \Z/2) =0$. Let $G$ be a ``contact
surface" to $F$, i.e., as schemes $G \cap F=2C$ for some curve $C$
on $F$, and suppose moreover, that $G$ has even multiplicity
$\alpha_i$ at every node $p_i$ of $F$ (this also allows $\alpha_i=0$
of course, whence $G$ does not pass through that particular
node). Assume that $G$ has even degree. Then $F$ splits in the double
cover of $\P^3$ branched in $G$.
\end{proposition}

\begin{proof}
The double cover of $\P^3$ is defined by adjoining a square root of
$T:=G/X_0^{\deg G}$ to the function field $k(\P^3) = k(X_1/X_0,
X_2/X_0 , X_3/X_0)$. Let $t \in k(F)$ be the restriction of $T$ to
$F$. We claim that $t$ viewed as an element of
\[
H^1 (k(F), \Z/2) = k(F)^{\times} / k(F)^{\times 2}
\]
is unramified with respect to every divisorial valuation $\nu$ of $k(F)$. Since we assumed that $H^1_{\mathrm{nr}} (k (F), \Z/2) =0$, this will imply that $t$ is a square, and the cover of $F$ determined by $t$ splits. By Proposition \ref{pPurityResult} we only have to check $\nu$'s corresponding to irreducible curves on a smooth model $\pi \colon \tilde{F}\to F$ where we have blown up all nodes $p_i$ to $(-2)$ curves $A_i$. Then the claim follows since
\[
\pi^{\ast} (2C) \equiv 2C' + \sum_i \alpha_i A_i
\]
where $C'$ is the strict transform of $C$ on $\tilde{F}$. See also \cite[proof of Prop. 2.6]{Cat81}. 
\end{proof}

\begin{remark}\label{rNonCartier}
If $G$, $F$ meet all the requirements of Proposition \ref{pSplitting} except that some $\alpha_i$ is not even, say $\alpha_i=1$ so that $G$ is smooth at $p_i$, then the cover of $F$ will \emph{not} split since $t$ will vanish to order $1$ along $A_i$ in that case. In particular, the intersection curve $C$ cannot locally analytically look like one line of a ruling in a cone at a node $p_i$ if we want the splitting.
\end{remark}

\begin{remark}\label{rNicestSituation}
In the nicest situation, the hypotheses of Proposition
\ref{pSplitting} will be satisfied in such a way that at a node $p$,
$C$ locally analytically looks like two lines of the ruling of a cone.
\end{remark}

\begin{example}\label{eHPT}
We will now analyze the example in \cite{HPT16}, which is a divisor
$Y_{\mathrm{HTP}}$ of bi-degree $(2,2)$ in $\P^2\times \P^3$, in light
of Proposition \ref{pSplitting}.  In \cite{HPT16}, the authors used
the structure of $Y_{\mathrm{HPT}}$ as a quadric surface fibration
over $\P^2$, given by the projection onto the first factor. We will
use its conic bundle structure over $\P^3$ given by projection onto
the second factor.  More precisely, $Y_{\mathrm{HPT}}$ is defined by
\begin{gather}\label{fHPTConic}
YZ\, S^2 + XZ\, T^2 + XY\, U^2 + (X^2+Y^2+Z^2-2(XY + XZ + YZ))\, V^2 = 0,
\end{gather}
where we denote homogeneous coordinates $(S:T:U:V)$ in $\P^3$ and $(X:Y:Z)$ in $\P^2$.

This conic bundle over $\P^3$ is defined, after rescaling the
coordinate $V \mapsto \sqrt{2}V$, by the graded matrix (up to a scalar
multiple)
\begin{gather}\label{fAsherMatrix}
\begin{pmatrix}
V^2 & U^2-V^2 & T^2-V^2 \\
U^2-V^2 & V^2 & S^2-V^2 \\
T^2-V^2 & S^2-V^2 & V^2 
\end{pmatrix}.
\end{gather}
The discriminant is a sextic surface $D \subset \P^3$ defined by the
determinant
\begin{gather}\label{fDiscriminantHPT}
4V^6 - 4(S^2+T^2+U^2) V^4 + (S^2+T^2+U^2)^2 V^2 - 2S^2T^2U^2 = 0
\end{gather}
which has two irreducible cubic surfaces as components $D_\pm$,
defined by
\begin{gather}\label{fDiscCompHPT}
2V^3 - V(S^2+T^2+U^2) \pm \sqrt{2}STU=0.
\end{gather}

Each component $D_\pm$ has four nodes and no other singular points,
hence up to projective equivalence, is isomorphic to the Cayley nodal
cubic surface.  In fact, given their equations, the surfaces $D_\pm$
are in the family of tetrahedral Goursat surfaces \cite{goursat},
which constitute one of the standard forms for the Cayley nodal cubic.
The nodes of the component $D_\pm$ are at the points
\begin{gather}\label{fNodesHPT}
\textstyle
(1:1:1:\pm\frac{1}{\sqrt{2}}), 
(1:-1:-1:\pm\frac{1}{\sqrt{2}}), 
(-1:1:-1:\pm\frac{1}{\sqrt{2}}), 
(-1:-1:1:\pm\frac{1}{\sqrt{2}}). 
\end{gather}
Over each node of the component $D_\pm$, the quadratic form $q$ has
rank 1. The only other points where the rank of $q$ drops to $1$ are
the six points
\begin{gather}\label{fAdditionalPoints}
\Sigma:= \{ 
(\pm\sqrt{2}: 0 :0 :1), 
(0:\pm\sqrt{2} : 0: 1), 
(0:0:\pm\sqrt{2} :1)\}.\nonumber
\end{gather}
Away from these $14$ points, $q$ has rank 2 on $D$.

The components of the discriminant meet in a curve $D_+ \cap D_-$,
which is a strict normal crossings curve of degree 9 in $\P^3$,
composed of an arrangement of 3 conics and 3 lines as in
Figure~\ref{fig:arrangement}.
\begin{figure}
\centering
    \begin{tikzpicture}[scale=1.8]
    \path[use as bounding box, color=white, draw] (1,{1/sqrt(3)}) +(-1.45,-1.3) rectangle +(1.45,1.55);
    \begin{pgfinterruptboundingbox}
    \clip (1,{1/sqrt(3)}) circle (1.5);
    \end{pgfinterruptboundingbox}
    \coordinate (A) at (-1,0);
    \coordinate (AA) at (3,0);
    \coordinate (B) at ({-1/2},{-sqrt(3)/2});
    \coordinate (BB) at ({3/2},{3*sqrt(3)/2});
    \coordinate (C) at (1/2,{3*sqrt(3)/2});
    \coordinate (CC) at (5/2,{-{sqrt(3)}/2});
    \draw[very thick, color=red, name path=lineA] (A)--(AA);
    \draw[very thick, color=red, name path=lineB] (B)--(BB);
    \draw[very thick, color=red, name path=lineC] (C)--(CC);
    \draw[very thick, color=blue, name path=circleA] (1,0) circle (2/3);
    \draw[very thick, color=blue, name path=circleB] (1/2,{sqrt(3)/2}) circle (2/3);
    \draw[very thick, color=blue, name path=circleC] (3/2,{sqrt(3)/2}) circle (2/3);
    \fill[color=red] (0,0) circle (1pt);
    \fill[color=red] (2,0) circle (1pt);
    \fill[color=red] (1,{sqrt(3)}) circle (1pt);
    \path [name intersections={of=circleA and lineA,name=lcA}];
    \fill[color=magenta] (lcA-1) circle (1pt);
    \fill[color=magenta] (lcA-2) circle (1pt);

    \path [name intersections={of=circleA and lineA,name=lcA}];
    \fill[color=magenta] (lcA-1) circle (1pt);
    \fill[color=magenta] (lcA-2) circle (1pt);

    \path [name intersections={of=circleB and lineB,name=lcB}];
    \fill[color=magenta] (lcB-1) circle (1pt);
    \fill[color=magenta] (lcB-2) circle (1pt);

    \path [name intersections={of=circleC and lineC,name=lcC}];
    \fill[color=magenta] (lcC-1) circle (1pt);
    \fill[color=magenta] (lcC-2) circle (1pt);
    \path [name intersections={of=circleA and circleB,name=ccAB}];
    \fill[color=blue] (ccAB-1) circle (1pt);
    \fill[color=blue] (ccAB-2) circle (1pt);

    \path [name intersections={of=circleB and circleC,name=ccBC}];
    \fill[color=blue] (ccBC-1) circle (1pt);
    \fill[color=blue] (ccBC-2) circle (1pt);

    \path [name intersections={of=circleA and circleC,name=ccAC}];
    \fill[color=blue] (ccAC-1) circle (1pt);
    \fill[color=blue] (ccAC-2) circle (1pt);    
    \end{tikzpicture}
\caption{The arrangement of components of the intersection of
irreducible components $D_+ \cap D_-$ of the discriminant of the conic
bundle associated to the example $Y_{\mathrm{HPT}}$ in \cite{HPT16}.}
\label{fig:arrangement}
\end{figure}
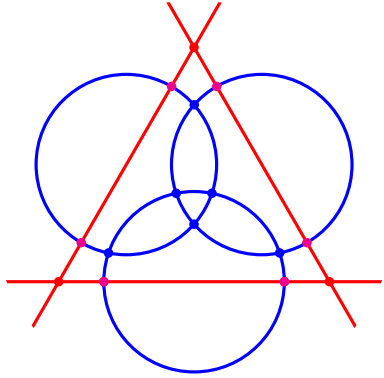
The equations of the components of $D_+ \cap D_-$ are:
\begin{align}\label{fIntersectionsUpstairs}
 \tilde{M}_1\colon ( U=S^2 + T^2 - 2V^2=0) \qquad {}& \qquad \tilde{L}_1\colon (U=V=0) \\\nonumber
   \tilde{M}_2\colon (T=S^2 + U^2 - 2V^2=0) \qquad {}& \qquad \tilde{L}_2 \colon (T=V=0) \\\nonumber
 \tilde{M}_3\colon  ( S=T^2 + U^2 - 2V^2=0) \qquad {}& \qquad \tilde{L}_3 \colon (S=V=0)\nonumber
\end{align}
Each two of the three conics intersect in two points, and the resulting set of six points coincides with $\Sigma$. 

Although we will verify it more easily in our geometric discussion
below, placing this example in the context of
Proposition~\ref{pSplitting}, the algebraically inclined reader can
verify already at this stage that
Theorem~\ref{tAuelPirutkaAlgebraicVersion} applies to
$Y_{\mathrm{HPT}}$, as follows.

By taking successive quotients of increasing minors, we can
diagonalize the quadratic form $q$ over $k(\P^3)$ (though still using
homogeneous coordinates) as
$$
q \sim \langle V^2, (-U^2+2U^2V^2)/V^2, D/(-U^2+2U^2V^2) \rangle
$$ 
where by abuse of notation, $D$ denotes the homogeneous equation for
the discriminant.  Hence, we have
$$
\alpha = (U^2-2V^2,D)
$$
in $\Br k(\P^3)$.  Hence over the generic point of each component
$D_\pm$ of $D$, we have residue $\gamma_\pm = \partial_{D\pm} \alpha =
(U^2-2V^2)$.  We know that each residue $\gamma_\pm$ is nontrivial.
Indeed, one verifies that $\gamma_\pm$ ramifies along valuations that
are centered at the isolated singular points of $D_\pm$, i.e., along
the exceptional divisors of a minimal resolution of $D_\pm$.

It is easy, but cumbersome, to check that $\gamma_\pm$ has no further
residues along components of $D_+ \cap D_-$ (which follows from the
fact that the quadratic form $q$ has rank 2 generically over each
component of $D_+ \cap D_-$) and that for each component $C$ of $D_+
\cap D_-$, the residue class is a square in the residue field $k(C)$.
Hence, Theorem~\ref{tAuelPirutkaAlgebraicVersion} gives that
$Y_{\mathrm{HPT}}$ has unramified Brauer group $\Z / 2\Z$.
\end{example}

\medskip

We now analyze the conic bundle $Y_{\mathrm{HPT}}$ in a more geometric
way, establishing the connection to Proposition~\ref{pSplitting}.

The first observation is that if we take another copy of $\P^3$ with
coordinates $X_0, X_1, X_2, X_3$ and consider the matrix
\begin{gather}\label{fMatrixCayleyCubic}
M = \begin{pmatrix}
X_0 & X_1 & X_2 \\
X_1 & X_0 & X_3\\
X_2 & X_3 & X_0
\end{pmatrix}
\end{gather}
then $M$ defines a linear determinantal conic bundle over that $\P^3$ with discriminant $\det M$ a Cayley cubic $F$ with nodes at
\[
\nu_0 = (1:1:1:1) , \; \nu_1 = (1:-1:-1:1) , \; \nu_2 = (1: 1:-1:-1), \; \nu_3= (1:-1:1:-1). 
\]
The conic bundle given by the matrix \eqref{fAsherMatrix} is the pull-back of this linear determinantal conic bundle via the degree $8$ cover
\begin{align}\label{fCoverConicBundle}
\varphi\colon \P^3_{(S:T:U:V)} &\to \P^3_{(X_0:X_1:X_2:X_3)} \\ 
 (S:T:U:V) & \mapsto (X_0:X_1:X_2:X_3) = ( V^2: U^2-V^2: T^2-V^2: S^2-V^2). \nonumber
\end{align}

The branch locus of this cover is given by a tetrahedron of planes in $\P^3$ given by 
\begin{align}\label{fBranchComponents}
G_0 &= \{ X_0=0\} \\\nonumber
G_1&= \{ X_0+X_1=0\}\\\nonumber
G_2 &=\{ X_0+X_2=0\}\\\nonumber
G_3 &=\{ X_0+X_3=0\}. \nonumber
\end{align}
We write $G= \bigcup_i G_i$. Let us give names to six lines on the Cayley cubic $F$

\begin{align}\label{fLinesOnCayley}
M_1&= \{ X_0+X_1=0,\, X_2+X_3=0\}\\\nonumber
M_2&= \{ X_0+X_2=0,\, X_1+X_3=0\} \\\nonumber
M_3&= \{ X_0+X_3=0,\, X_2+X_1=0\} \\\nonumber
& \\\nonumber
L_1&=\{ X_0=X_1=0\} \\\nonumber
L_2&=\{ X_0=X_2=0\} \\\nonumber
L_3&=\{ X_0=X_3=0\} \nonumber
\end{align}
and write

\[
L = \bigcup_i L_i, \quad M= \bigcup_j M_j. 
\]

Then $L$ and $M$ are two triangles of lines in $F$ that are ``circumscribed around each other", in the sense that $L_i$ meets $M_i$ in a point different from the vertices of $M$, and $L_i$ does not meet $M_j$ for $i\neq j$. Moreover, the nodes $\nu_1, \nu_2, \nu_3$ form the vertices of the triangle $M$. We have the following scheme-theoretic intersections
\begin{align}\label{fIntersections}
G_0 \cap F &= L \\\nonumber
G_i\cap F &= 2M_i + L_i, \; i=1,2,3\\\nonumber
G \cap F &= 2L + 2M\nonumber
\end{align}

So the $G_i$, $i=1,2,3$, are tangent to $F$ in $M_i$, and $G$ itself is singular along $L$, $G_i\cap G_0=L_i$, $i=1,2,3$. Note that the curve $C: =L+M$ is Cartier everywhere, even at the nodes. The node $\nu_0=(1:1:1:1)$ is not in $G$ at all. 

\medskip

In other words, $F$, $G$, and $C$ verify all the hypotheses of Proposition \ref{pSplitting}! The eight to one cover $\varphi$ in (\ref{fCoverConicBundle}) factors into a double cover to which Proposition \ref{pSplitting} applies, and a residual four to one cover. This explains the splitting of the discriminant conceptually for the example $Y_{\mathrm{HPT}}$. 

The eight singular points of $D_+$ and $D_{-}$ (both Cayley cubics) are the preimages under $\varphi$ of $\nu_0$. In fact, the cover is \'{e}tale locally above $\nu_0$. The following formulas hold for the (reduced, set-theoretic) preimages:

\begin{align}\label{fPreimages}
\varphi^{-1} (L_i) & = \tilde{L}_i\\\nonumber
\varphi^{-1} (M_i) & = \tilde{M}_i. \nonumber
\end{align} 

We have 
\[
\varphi^{-1}(\{ \nu_1, \nu_2, \nu_3\}) = \Sigma .
\]

Let us now verify that the double covers of the curves $\tilde{L}_i$ and $\tilde{M}_j$ induced by the conic bundle given by (\ref{fAsherMatrix}) decompose into two components. Indeed, look at the double covers of the $L_i$ induced by the conic bundle given by (\ref{fMatrixCayleyCubic}) first. Then these already split into two components, as is easy to see. For example, taking the line $L_1$ with homogeneous coordinates $X_2, X_3$, and fiber coordinates $(a:b:c)$ in the trivial $\P^2$ bundle that the conic bundle given by (\ref{fAsherMatrix}) naturally embeds into, the preimage of $L_1$ decomposes as
\[
c=0, \quad X_2a+X_3b=0.
\]
Similarly for $L_2, L_3$. So also the double covers of the curves $\tilde{L}_i$ decompose. The double covers of the curves $M_j$ on the contrary are irreducible conics $M_j^{\sharp}$, the covers $M_j^{\sharp} \to M_j$ being branched in the two nodes of $F$ lying on $M_j$.  However, if we pull-back the cover $M_j^{\sharp} \to M_j$ via the cover $\tilde{M}_j \to M_j$, then it becomes reducible (since $\tilde{M}_j$ is square isomorphic to $M_j^{\sharp}$ over $M_j$). 
So all the hypotheses of Corollary \ref{cAuelPirutkaMoreGeometric}, including the ``splitting condition" for the curves arising as irreducible components of some $S_i\cap S_j$, are verified. So we see again that the unramified Brauer group of $Y_{\mathrm{HPT}}$ is equal to $\Z/2\Z$. 

\medskip

In \cite{HPT16}, the authors show that $Y_{\mathrm{HPT}}$ has a Chow universally trivial resolution of singularities, by an explicit computation. The results of Section \ref{sSingularities} give a new streamlined proof of this result. Using \cite{Voi15} and \cite{CT-P16}, one obtains that the very general divisor of bi-degree $(2,2)$ in $\P^2\times \P^3$ is not stably rational. On the other hand, some such hypersurfaces, even smooth ones, are shown to be rational in \cite{HPT16}.

\begin{remark}\label{rFailure}
The difficulty in using this approach, or, more precisely,
Proposition~\ref{pSplitting}, for the construction of new examples to
which Theorem \ref{tAuelPirutkaAlgebraicVersion} applies is that the
double cover $B$ of $\P^3$ branched in $G$ is usually both nonrational
and has nontrivial $\Het^3(B, \Z /2)$. In cases where $B$ is
at least unirational, one can pull back further to a rational $B'$
dominating $B$, but also this will usually have $\Het^3(B',
\Z /2)$ nontrivial.
\end{remark}


\section{Reducibility of the discriminant: 2nd method}\label{sCatanese}


There is another construction of conic bundles to which Corollary \ref{cAuelPirutkaMoreGeometric} potentially applies, again using the theory of contact of surfaces. The advantage of this method is that it works over the base $B=\P^3$ and that it produces conic bundles of graded-free types with reducible discriminant surfaces directly, and such that the conics will generically be two distinct lines over the intersections of discriminant components. The hard condition that one then still has to ensure somehow (e.g., by adjusting the free parameters in the construction) is the splitting condition on the covers of the curves that make up the irreducible components of the intersection of two discriminant surfaces. But also this can be translated entirely into projective geometry of the configuration, and we will deal with it at the end of this Section.

\medskip

\begin{proposition}\label{pCombine}
Consider symmetric matrices over $\P^3$
\[
	A = 
	\begin{pmatrix}
		a_{0,0} & a_{0,1} & a_{0,2} \\
		a_{0,1} & a_{1,1} & a_{1,2} \\
		a_{0,2} & a_{1,2} & a_{2,2} \\
	\end{pmatrix},
	\quad\quad
	B = 
	\begin{pmatrix}
		b & c \\
		c & d
	\end{pmatrix}
\]
defining symmetric maps between graded-free vector bundles. Let 
\[
	N =
	\begin{pmatrix}
		c^2a_{0,0} - b \det A & ca_{0,1} & ca_{0,2} \\
		ca_{0,1} & a_{1,1} &  a_{1,2} \\
		ca_{0,2} & a_{1,2} & a_{2,2}   \\
	\end{pmatrix}.
\]
If in this situation 
\[
	d = \det 	\begin{pmatrix}
		a_{1,1} & a_{1,2} \\
		a_{1,2} & a_{2,2} \\
	\end{pmatrix}
\]
then $N$ also gives a symmetric map between graded-free vector bundles and 
\[
	\det N = -(\det A)(\det B).
\]
\end{proposition}

\begin{proof}
First notice that 
\begin{align*}
2\deg (c) + \deg (a_{0,0}) &= \deg (b) + \deg (d) + \deg (a_{0,0})\\
   & = \deg (b) + \deg (a_{1,1} ) + \deg (a_{2,2}) + \deg (a_{0,0}) \\
   &= \deg (b) + \deg (\det (A)). 
\end{align*}
Then evaluate $\det N$ and compare.
\end{proof}

\begin{remark}
For the interested reader we sketch how the above construction was found. 
Even though the concepts are not used in the proof, this construction relies on matrix factorizations and Catanese's theory of contact of surfaces \cite{Cat81}: 

The minimal free resolution of a coherent sheaf on a hypersurface $X =
\{f=0\} \subset \P^n$ over the coordinate ring of $X$ becomes periodic
after a finite number of steps. If the sheaf is arithmetically
Cohen--Macaulay (ACM) with support equal to $X$, the resolution is
periodic. The differentials are given by square matrices $P$ resp.\
$Q$ corresponding to maps from $F$ to $G$ resp.\ $G$ to $F$ for some
graded free modules $F$ and $G$, with $PQ = f\id_G$ and $QP=f
\id_F$. Furthermore the determinants of $P$ and $Q$ vanish on $X$. The
pair $(P,Q)$ with the above properties is called a {\sl matrix
factorization} of $f$ \cite[Thm. 6.1]{Ei80}.

Dolgachev \cite[Section 4.2]{Dol12} observes that one obtains
symmetric matrices in this way if one starts with an arithmetically
Cohen--Macaulay symmetric sheaf.  So our problem of finding a
symmetric matrix with given reducible determinant $X$ can be reduced
to finding an appropriate sheaf on $X$.

On the other hand, Catanese observed that for a symmetric graded $n
\times n$ matrix each diagonal $(n-1) \times (n-1)$ minor defines a
contact surface to the determinant of the matrix. Furthermore the
square root of the contact curve is defined by the $(n-1) \times
(n-1)$ minors of the $(n-1) \times n$ matrix obtained by deleting the
line that is not involved in the minor defining the contact
surface. In our construction above
\[
d = \det 	\begin{pmatrix}
		a_{1,1} & a_{1,2} \\
		a_{1,2} & a_{2,2} \\
	\end{pmatrix}
\]
is a contact surface to both $\det A$ and $\det B$. The contact curves are defined by the $2 \times 2$ minors of
\[
	\begin{pmatrix}
		a_{0,1} & a_{1,1} & a_{1,2} \\
		a_{0,2} & a_{1,2} & a_{2,2} \\
	\end{pmatrix},
\]
and the $1 \times 1$ minors of
\[
	\begin{pmatrix}
		d & c \\
	\end{pmatrix}.
\]
The ideal sheaves of these curves are ACM (since they are determinantal) and symmetric (since they are contact curves).

Notice now that $d$ is also contact to $(\det A)(\det B)$. Furthermore the contact curve is the union of the two contact curves above.
If this union is also $ACM$ we can obtain a symmetric matrix $N$ whose determinant vanishes on $(\det A)(\det B)$ via matrix factorization. 

In our case the union of the curves is defined by
\[
	\begin{pmatrix}
		ca_{0,1} & a_{1,1} &  a_{1,2} \\
		ca_{0,2} & a_{1,2} & a_{2,2}   \\
	\end{pmatrix}
\]
Indeed, if $c$ is nonzero, we obtain the equations of the first
curve. If $c=0$ two of the minors vanish automatically and the third
is just $d$. So we obtain $d=c=0$ as the second component. This shows
that the union of contact curves is again ACM and we obtain the above
formula via matrix factorization.

In a certain sense this is a generalization of the construction of
Artin and Mumford in \cite{AM72} to $\P^3$.
\end{remark}

Note that $N$ defines a conic bundle of graded-free type if the rank of $N$ is never zero in a point of $\P^3$. 

\begin{remark} \label{rConicBundlesSame} Notice that if in the above
construction $A$, $B$ and $N$ define conic bundles, then the
restriction of the conic bundle defined by $N$ to $\det A$ is
birationally the same as the one defined by $A$.
\end{remark}

\begin{remark}
In order to apply our Theorem \ref{tAuelPirutkaAlgebraicVersion}, or
rather Corollary \ref{cAuelPirutkaMoreGeometric}, to the situation
above we must find $A$ and $B$ such that
\begin{enumerate}
\item $\det A$ and $\det B$ are irreducible (this is an open condition)
\item $\det A$ and $\det B$ are smooth in the intersection curve
$\barD = \{\det A = \det B = 0\}$ (this is an open condition)
\item the double cover of $\det A$ and $\det B$ induced by $N$ is non trivial (this is also an open condition)
\item
$N$ has rank two generically on each component of $\barD$. 
\item the double cover of the intersection curve $\barD$ induced by $N$ is trivial (this is a closed condition)
\end{enumerate}
The hard part here is the last condition. In the next section we will
show how one can satisfy this closed condition via an appropriate
construction.  The open conditions will then be checked by a computer
program for a one example.
\end{remark}


\section{Triviality of the conic bundle on the intersection curve}\label{sIntersectionCurve}


In order to apply Theorem \ref{tAuelPirutkaAlgebraicVersion} to the
situation of Section~\ref{sCatanese} we must ensure that the double
cover of the intersection curve $\barD= \{\det A=\det B=0\}$ induced
by $N$ is trivial. For this we have the following geometric
(sufficient) condition:

\begin{proposition} \label{pRationalComponents}
In the notation of Proposition~\ref{pCombine} let 
$$\barD = \{\det A=\det B=0\} \subset \P^3$$ 
be the intersection curve of the two discriminant components. If all
components of $\barD$ are rational and do not intersect the rank $1$
locus of $A$, and, moreover, $N$ has rank $2$ generically on each
component of $\barD$, then the double cover of each component of
$\barD$ induced by $N$ is trivial.
\end{proposition}

\begin{proof}
By Remark \ref{rConicBundlesSame} the double cover of $\barD$ induced
by $N$ is birationally the same as the one induced by $A$. Since
$\barD$ does not intersect the rank $1$ locus of $A$ this double cover
is \'etale. Since there are no nontrivial \'etale double covers of
$\P^1$ and $\barD$ consists of rational components, the double cover
induced by $A$, and with it the one induced by $N$, is trivial.
\end{proof}

For the remainder of this Section we restrict to the case where all
$a_{i,j}$ are linear and $\det A$ is the Cayley cubic. We can change
coordinates, or, equivalently, find an invertible matrix $S$ such that
\[
	SAS^t = \begin{pmatrix}
		x_0 & x_1 & x_2 \\
		x_1 & x_0 & x_3 \\
		x_2 & x_3 & x_0
	      \end{pmatrix}
\]
For our construction we will use the fact that the Cayley cubic is rational:

\begin{proposition}
Let $L_1,\dots,L_4$ be $4$ general linear forms defining 4 general
lines in $\P^2$ intersecting in $6$ distinct points. Consider the
cubic polynomials
\[
	Y_i = \prod_{i\not=j} L_j
\]
and
\begin{align*}
	X_0 &=  -Y_0+Y_1+Y_2+Y_3\\
	X_1 &=  -Y_0-Y_1-Y_2+Y_3 \\
	X_2 &=   Y_0-Y_1+Y_2+Y_3 \\
	X_3 &=   Y_0+Y_1-Y_2+Y_3 \\
\end{align*}
Then the image of $\P^2$ under the rational map $\varphi \colon \P^2
\dashrightarrow \P^3$ defined by the linear system $|\langle
X_0,X_1,X_2,X_3 \rangle |$ is the Cayley cubic.
\end{proposition}
	      
\begin{proof}
Set $x_i=X_i$ in $SAS^t$ then the evaluation of the determinant gives zero.
\end{proof}

\begin{remark}
Recall the following facts from classical algebraic geometry:
\begin{enumerate}
\item The Cayley cubic has $4$ nodes. They form the rank $1$ locus of $A$.
\item The four lines $L_1,\dots,L_4$ are contracted by $\varphi$. Their images are the $4$ nodes. 
\item The 6 base points are blown up and their images are $6$ lines in $\P^3$. These $6$ lines form a tetrahedron with the $4$ nodes as vertices.
\end{enumerate}
\end{remark}

\begin{notation}
Let $\sigma \colon \blowup \to \P^2$ be the blowup of $\P^2$ in the
$6$ base points above. With this we have the following diagram
\begin{center}
\mbox{
\xymatrix{
	& \blowup  \ar[dl]_{\sigma} \ar[dr]^{\pi}& & \\
	\P^2 \ar@{-->}[rr]^{\varphi} && X_3 \subset \P^3 
	}
	}
\end{center}
where $X_3 \subset \P^3$ denotes the Cayley cubic. If $C \subset \P^2$
is a plane curve, we denote by $\tildeC \subset \blowup$ its strict
transform and by
$$\barC := \pi(\tildeC) \subset X_3 \subset \P^3$$
 the image of $\tildeC$ in $\P^3$. Furthermore, denote by $E_{i,j}\subset \blowup$ the exceptional divisor over the
intersection point of $L_i$ and $L_j$, and by $H$ the class of the pull back of a line in $\P^2$ to $\blowup$. 
\end{notation}

We are interested in curves on the Cayley cubic that do not intersect
the nodes.

\begin{lemma} \label{lNoNodes} Let $\tildeC\subset \blowup$ be the
strict transform of a curve $C$ in $\P^2$ not containing any of the
$L_i$ as components, and suppose its class is
\[
	\tildeC \equiv \alpha H - \sum_{i< j} \beta_{i,j} E_{i,j}.
\]
Then the image $\barC = \pi(\tildeC) \subset \P^3$ avoids the nodes of
the Cayley cubic if and only if $\beta_{i,j} = \beta_{k,l}$ for all
indices with $\{i,j,k,l\} = \{1,2,3,4\}$ and $\alpha = \sum_j
\beta_{i,j}$ for every $i$.
\end{lemma}

\begin{proof}
Since the preimage of the nodes are the lines $L_i$ we want
$\tildeC. \tildeLi= 0$ for all $i$ where $\tildeLi$ is the strict
transform of $L_i$ on the blow up.  This gives the following linear
system of equations
\[
	\begin{pmatrix}
		1 & -1 & 0 & -1 & 0 & -1 & 0 \\
		1 & -1 & 0 & 0 & -1 & 0 & -1 \\
		1 & 0 & -1 & -1 & 0 & 0 & -1 \\
		1 & 0 & -1 & 0 & -1 & -1 & 0
	\end{pmatrix} \cdot
	(\alpha, \beta_{1,2},\beta_{3,4},\beta_{1,3},\beta_{2,4},\beta_{1,4},\beta_{2,3})^t = 0
\]
The solution of this system is the one claimed above.
\end{proof}

\begin{definition}
We call a curve $C \subset \P^2$ of type $(b_1,b_2,b_3)$ if its strict
transform has class
\[
	\tildeC \equiv (b_1+b_2+b_3) H - b_1(E_{1,4}+E_{2,3}) - b_2(E_{2,4}+E_{1,3}) - b_3(E_{3,4}+E_{1,2})
\]
If $C$ does not contain any of the lines $L_i$ as component, then the
image $\barC \subset \P^3$ of such a curve avoids the nodes of the
Cayley cubic by Lemma \ref{lNoNodes}.
\end{definition}

We collect some numerical facts about these curves:

\begin{lemma}
Let $C \subset \P^2$ be a curve of type $(b_1,b_2,b_3)$ and $\tildeC$
its strict transform and $\barC \subset \P^3$ its image. Then
\begin{enumerate}
\item The degree of $\barC$ is $\deg (\barC)=b_1+b_2+b_3$.
\item The arithmetic genus of $\barC$ is $g_a=\binom{b_1+b_2+b_3}{2}-(b_1^2+b_2^2+b_3^2)+1$.
\item The expected number of moduli of $C$ is $\deg (\barC)+g_a$.
\end{enumerate}
\end{lemma}

\begin{proof}
For the first two items we work on $\blowup$.  The linear system of
$\varphi$ has class $-K = 3H-\sum E_{i,j}$ there, i.e., it consists of
curves of type $(1,1,1)$.  This is also the anticanonical system. We
have
\[
	\deg \barC = -K.\tildeC = 3(b_1+b_2+b_3) - 2 b_1 - 2 b_2 - 2 b_3 = b_1+b_2+b_3.
\]
The arithmetic genus of $\barC$ is given by the adjunction formula
\[
	2g_a-2 = K.\tildeC+\tildeC^2 = -b_1-b_2-b_3 + (b_1+b_2+b_3)^2-2b_1^2-2b_2^2-2b_3^2.
\]
For the number of moduli, we work with plane curves.  The dimension of
the space of degree $b_1+b_2+b_3$ curves in $\P^2$ is $\binom{b_1+b_2+b_3+2}{2}$, the number of conditions for a $b_i$ fold point is $\binom{b_i+1}{2}$. Therefore the expected number of moduli is
\[
	\binom{b_1+b_2+b_3+2}{2} - \sum_{i=1}^3 \binom{b_i+1}{2}
\]
which simplifies to the formula above.
\end{proof}

\begin{example}
We have for examples:

\begin{center}
\renewcommand{\arraystretch}{1.25}
\begin{tabular}{|c|c|}
\hline
type & image in $\P^3$ \\
\hline
$(1,0,0)$ & a line \\
$(1,1,0)$ & a plane conic \\
$(1,1,1)$ & a plane cubic \\
$(2,1,1)$ & an elliptic normal curve of degree $4$ \\
$(2,2,2)$ & a canonical curve, i.e., degree $6$ and genus $4$  \\
$(1,2,3)$ & a sextic curve of genus 2 \\
\hline
\end{tabular}
\end{center}

\end{example}

Let us now look at a contact quadric to the Cayley surface.
\begin{proposition}
Let $Q \subset \P^3$ be a contact quadric defined by a generalized $2
\times 2$ diagonal minor of $A$. Then there exists a line $\Lcontact
\subset \P^2$ such that the transform $\sigma_* \pi^* (Q \cap X_3)$ of
$Q$ on $\P^2$ is
\[
	q = \Lcontact^2 + L_1 + L_2 + L_3 + L_4.
\]
\end{proposition}

\begin{proof}
The contact quadric passes through all nodes of $X_3$ (it is one of
the minors defining the ideal of the nodes), so its transform contains
the lines $L_1,\dots,L_4$. Outside of the nodes the contact quadric
intersects the Cayley cubic with multiplicity $2$. It follows that the
transform has the form
\[
	\Lcontact^2 + L_1 + \dots + L_4.
\]
Since the transform of any quadric is of degree $6$ it follows that
$\Lcontact$ must be a line.
\end{proof}

Notice that the transform of $\{\det B = 0 \}$ on $\P^2$ is just the
transform of the intersection curve $\barD$ on $\P^2$. To keep with
our convention, we denote this by $D$.  In other words, on $\P^2$, we
have that $D$ is the determinant of the matrix obtained by forming the
transforms of all the entries in $B$. In view of
Proposition~\ref{pRationalComponents}, we would like $D$ to be a union
of rational curves. The idea of the construction is now to start with
such a $D$ and then try to write it as a determinant. Again we would
like to mimic the construction of Artin and Mumford. For this we need
a slight generalization of their method to the case where the contact
curve is not reduced. For this we need the following technical lemma.

\begin{lemma} \label{lSqrt} Let $D$ be a curve of type $(d,d,d)$ with
$d\ge 4$ even and $\frac{3d}{2}$ ordinary nodes on $\Lcontact$. Let
$f$ be a generator of the ideal of $D$ and $s$ a generator of the
ideal of $\Lcontact$.  Suppose that $\Lcontact$ does not pass through
any of the base points and that $D$ avoids the intersection points of
$L_c$ with the exceptional lines. Let $Z\subset \P^2$ be the subscheme
consisting of all the base points with multiplicity
$\frac{d}{2}-2$. Assume that the natural map
\begin{gather}\label{fGeneralPosition}
H^0 \left( \P^2, \OO_{\P^2}\left( \frac{3d}{2}-6 \right) \right) \to
H^0 \left( \P^2, \OO_{Z} \right)
\end{gather}
is surjective. 

Then 
there exist a polynomial $g$ on $\P^2$ such that
\begin{enumerate}
\item $f \equiv g^2 \mod s^2$
\item  the curve $\sqrt{D}$ defined by $\{g=0\}$ is of type  $(\frac{d}{2},\frac{d}{2},\frac{d}{2})$.
\end{enumerate}
\end{lemma}

\begin{proof}
Choose homogeneous coordinates $u,v,s$ in $\P^2$. Since $D$ has only
ordinary nodes on $\Lcontact = \{s=0\}$, hence, in particular,
intersects $\Lcontact$ in a divisor that is divisible by $2$, there
exist a polynomial $g_0\in k[u,v,s]$ with
\[
	g_0^2 \equiv f \mod s.
\]
More precisely, we choose $g_0$ such that it vanishes at the nodes of
$D$ on $\Lcontact$ and has multiplicity $\frac{d}{2}$ in all base
points. This is clearly possible for $d\ge 4$ since an ordinary
multiple point of order $e$ imposes $e(e+1)/2$ conditions on plane
curves, and $g_0$ has degree $3d/2$. We therefore have a polynomial
$f_1 \in K[u,v,s]$ such that
\[
	f - g_0^2 = f_1 s.
\]
Taking the derivative with respect to $s$ we get
\[
	\frac{df}{ds}- 2g_0\frac{dg_0}{ds} = f_1 + \frac{d f_1}{ds} s.
\]
For every point $P \in L_c \cap D$ all derivatives of $f$ vanish
(since $D$ has a node there). Also $g_0$ vanishes at all such points
by construction. Therefore the equation above also gives $f_1(P) =
0$. This implies that $g_0$ divides $f_1$ modulo $s$, i.e., there
exists a $g_1$ such that
\[
	2g_0g_1 \equiv f_1 \mod s.
\]
We obtain
\[
	(g_0 + g_1 s)^2 \equiv g_0^2 + 2g_0g_1 s \equiv g_0^2 + f_1s \equiv f  \mod s^2.
\]
We now want to find a $g_2 \in K[u,v,s]$ such that 
\[
	g = g_0 + g_1 s + g_2 s^2 
\]
defines a curve of type
$(\frac{d}{2},\frac{d}{2},\frac{d}{2})$. Notice that this leads to an
affine linear system of equations for the coefficients of $g_2$. To
prove the solvability of this system we have to analyze the geometric
situation in more detail.

Firstly notice that $\{f_1 = 0\}$ is a curve of degree $3d-1$ that
passes with multiplicity $d$ through each base point (since $L_c$ does
not contain any of the base points). Now there are $3$ base points on
each exceptional line. It follows by Bezout's theorem that $\{f_1 =
0\}$ contains all $4$ exceptional lines as components. We can
therefore write
\[
	f_1 = f_1' l_1 l_2 l_3 l_4
\]
where $l_i$ is an equation for $L_i$.  Furthermore, since none of the
exceptional lines pass through any of the nodes of $D$, we have that
$g_0$ divides not only $f_1$, but also $f_1'$ modulo $s$. It follows
that there is a polynomial $g_1'$ with
\[
	2g_0g_1' \equiv f_1' \mod s
\]
and 
\[
	g_1 = g_1'l_1l_2l_3l_4.
\]
We have $\deg g_1 = \deg f_1 - \deg g_0 = 3d-1 - \frac{3d}{2} =
\frac{3d}{2} -1$ and therefore
\[
	\deg g_1' = \frac{3d}{2}-5.
\]
Now, the surjectivity of the map (\ref{fGeneralPosition}) implies the
existence of a $g_2'$ of degree $\frac{3d}{2}-6$ such that
\[
	g_1' + sg_2'
\]
has multiplicity $\frac{d}{2}-2$ in each base point. With $g_2 :=
g_2'l_1l_2l_3l_4$ we obtain that
\[
	\{g_1 + sg_2 = 0 \}
\]
passes through all base points with multiplicity $\frac{d}{2}$. Since
the same is true for $g_0$ we get that
\[
	g = g_0 + sg_1 + s^2 g_2
\]
defines a curve of type $(\frac{d}{2},\frac{d}{2},\frac{d}{2})$.
\end{proof}

With this, we get an instance of our generalized version of the
Artin--Mumford method.

\begin{proposition}\label{pDeterminantalRep}
Let $D$ be a curve of type $(d,d,d)$ with $d\ge 4$ even.  Assume that
$D$ has $\frac{3d}{2}$ ordinary nodes on $\Lcontact$, $\Lcontact$
contains none of the base points, $D$ avoids the intersection points
of $L_c$ with the exceptional lines, and that the map
(\ref{fGeneralPosition}) is surjective.  Then there exists a matrix
\[
	B = \begin{pmatrix} q & r \\ r & t \end{pmatrix}
\]
with $\{q=0\}$ the transform of the contact quadric $Q$, $\{r=0\}$
defining $\sqrt{D}$, and $\{t=0\}$ of type $(d-2,d-2,d-2)$, such that
$D$ is defined by $\det B$.
\end{proposition}

\begin{proof}
Let $f$ be a defining equation of $D$. By Lemma \ref{lSqrt} there
exists a curve $\sqrt{D}$ with defining equation $r=0$ such that $f
\equiv r^2 \mod s^2$. Therefore $f-r^2$ is divisible by $s^2$. Now
$f-r^2$ vanishes on each line $L_i$ with multiplicity $d$ in the three
base points that lie on $L_i$. Furthermore $f-r^2$ vanishes with
multiplicity $2$ on the intersection $\Lcontact \cap L_i$.  So $f-r^2$
vanishes with multiplicity at least $3d+2$ on $L_i$. Bezout's theorem
implies then that $f-r^2$ vanishes also on $L_i$.  In total $f-r^2$
vanishes on $\{q =0\} = \Lcontact^2 + L_1 + \dots + L_4$ and is
therefore divisible by $q$. Set
\[
	t := -\frac{f-r^2}{q}
\]
with this we get
\[
	-f = qt-r^2 = \det \begin{pmatrix} q & r \\ r & t \end{pmatrix}.
\]
\end{proof}

\begin{lemma}\label{lGenericityCondition}
The map (\ref{fGeneralPosition}) is surjective for $d=6$.
\end{lemma}
\begin{proof}
For $d=6$ the scheme $Z$ is the union of all base points
$P_1,\dots,P_6$ with multiplicity $1$ and the map
(\ref{fGeneralPosition}) is
\[
H^0 \left( \P^2, \OO_{\P^2}\left( 3 \right) \right) \to H^0 \left( \P^2, \OO_{Z} \right).
\]
For the surjectivity of this map we construct cubics $C_i$ that pass
through all $P_j$ with $j \not= i$ but not through $P_i$.

For this, notice that there is no quadric that passes through all six
$P_i$. Indeed, assuming the contrary we would get a quadric $Q$ that
passes through $3$ points on every exceptional line and must therefore
contain all $4$ such lines as a factor, which is a contradiction.

For each $i \in \{1,\dots,6\}$ there exists a quadric $Q_i \not=0$
passing through the five $P_j$ with $j \not= i$.  Since there is no
$Q$ through all six base points, we have $Q_i(P_i) \not=0$. Now choose
a line that does not pass through $P_i$ and we get cubics $C_i =
L_iQ_i$ with the desired properties.
\end{proof}

The next problem in our construction is to find curves $D$ of type
$(d,d,d)$ with all components rational.

\begin{remark}
The existence of such curves $D$ of type $(d,d,d)$, with components
rational, and with $\frac{3d}{2}$ nodes on $\Lcontact$ is expected.
Indeed, the arithmetic genus $g_a$ of the image of $\barD$ in $\P^3$
is
\[
	g_a = \binom{3d}{2} - 3 d^2 + 1 = \frac{3d(3d-1)}{2} - 3d^2 + 1 = \frac{3}{2} d^2 - \frac{3}{2}d + 1 = 3\binom{d}{2} + 1,
\]
in particular $g_a > \frac{3d}{2}$. For $D$ to be rational we need it
to have $g_a$ nodes. This poses $g_a$ conditions.  Furthermore,
$\frac{3d}{2}$ of them should lie on $\Lcontact$. This poses a further
$\frac{3d}{2}$ conditions. So we have $\frac{3d}{2} + g_a$ conditions
and $3d+g_a$ moduli. So we expect such curves to exist.

Unfortunately, this is not enough to apply Theorem
\ref{tAuelPirutkaAlgebraicVersion}. For this we must also show that a
number of open conditions are satisfied. We propose to do this by
constructing a concrete example over a finite field $\F_p$ along the
lines suggested so far in this Section and then check the open
conditions for this example.

Now, finding a rational curve as described above explicitly is hard,
since the conditions above are highly nonlinear. For example, having a
node {\sl somewhere} means that a certain discriminant of high degree
in the coefficients of $D$ vanishes. This is a highly nonlinear
codimension $1$ condition. Having a node {\sl at a given point} on the
other hand is a linear codimension $3$ condition. So one might try to
construct such a curve by prescribing $g_a$ nodes {\sl at given
points} (some of them on $\Lcontact$). Unfortunately this poses
$$3g_a > g_a + 3d$$ conditions, which is larger than the number of
moduli.

So we must choose our curves more carefully, which takes up the
remainder of this section.
\end{remark}

\begin{construction} \label{c666}
Consider the case $d=6$ with reducible $D = D_1+D_2+D_3$ and $D_i$ of type $(1,2,3), (2,3,1)$ and $(3,1,2)$ respectively.

\begin{enumerate}
\item Choose points $P_1,\dots,P_6$ and $Q_1$ on $\Lcontact$.
\item Choose a curve  $D_1$ of type $(1,2,3)$ with nodes at $P_1$ and $P_2$ and vanishing at $Q_1$. This is possible since the
number of projective moduli of such curves is $d+g_a-1 = 6+2-1 = 7$ and the number of conditions imposed is $3+3+1 = 7$.
So generically there is only one such curve. 
\item  $D_1$ has degree $6$ and of the $6$ intersection points with $\Lcontact$ we have prescribed $5$ so far. Let $Q_2$ be the
remaining intersection point.
\item Choose a curve $D_2$ of type $(2,3,1)$ with nodes at $P_3$ and $P_4$ also passing through $Q_1$. Again there is generically 
one such curve. 
\item Let $Q_3$ be the remaining intersection point of $D_2$ with $\Lcontact$.
\item Choose a curve $D_3$ of type $(3,1,2)$ with nodes $P_5$ and $P_6$ and passing through $Q_2$.
\item Let $Q_4$ be the remaining intersection point of $D_3$ with $\Lcontact$.
\end{enumerate}

We can summarize the construction so far in the following table:

\begin{center}
\begin{tabular}{|c|c|c|c|c|c|c|c|c|c|c|}
\hline
 & $P_1$ & $P_2$ & $P_3$ & $P_4$ & $P_5$ & $P_6$ & $Q_1$ & $Q_2$ & $Q_3$ & $Q_4$ \\
\hline
$D_1$ & 2 & 2 &&&&& 1 & 1&&\\
\hline
$D_2$ &&& 2 & 2&&& 1 && 1 &\\
\hline
$D_3$ &&&&& 2& 2&  & 1&  & 1\\
\hline\hline
$D$ &2&2&2&2& 2& 2& 2 & 2 & 1& 1\\
\hline
\end{tabular}
\end{center}

Now if $Q_3 = Q_4$ this gives a curve $D$ with $9$ nodes on $\Lcontact$. This is at most a codimension $1$ condition. 
Furthermore for each $i \in \{1,2,3\}$ the curve $D_i$ is of arithmetic genus $g_a = 2$ and therefore of geometric genus $0$.
\end{construction}

\begin{remark}
For reasons not clear to us, the condition $Q_3 = Q_4$ was automatically satisfied  in all examples we tried.
\end{remark}

\begin{proposition}\label{pConicBundle}
There exists a conic bundle $Y \to \P^3$, defined over a finite field
$k_0= \F_p$, $p=10007$, defined by a homogeneous $3 \times 3$ matrix
with entries of degrees
\[
	\begin{pmatrix}
	    7 & 4 & 4 \\
	    4 & 1 & 1 \\
	    4 & 1 & 1 
	 \end{pmatrix}
\]
such that Corollary \ref{cAuelPirutkaMoreGeometric} predicts a
nontrivial unramified Brauer class for the base change of $Y$ to the
closure $k$ of $k_0$, hence $Y$ is not stably rational (over $k$).
\end{proposition}

\begin{proof}
Construct a curve $D=D_1+D_2+D_3$ as in Construction \ref{c666} over
the finite field $k_0$ using a computer algebra program. Denote by
$\barD = \barD_1 + \barD_2 + \barD_3$ the image of the strict
transformation of the previous curves in $\P^3$.

Calculate a matrix representation $\det B$ for $D$ using Proposition
\ref{pDeterminantalRep}. Find a preimage $\barB$ of $B$ in $\P^3$.
The determinant of $\barB$ defines a sextic hypersurface $X_6 \subset
\P^3$. Use Proposition \ref{pCombine} to construct a matrix $N$ with
the degrees claimed.  Then check the following:

\begin{enumerate}
\item $X_6$ is irreducible. We do this by checking that the singular locus is finite.
\item $X_6$ is smooth along $\barD$.
\item The Cayley cubic is smooth along $\barD$.
\item
The rank $1$ locus of $N$ is finite.
\item
The rank $0$ locus of $N$ is empty.
\item The curves $\barD_i$ are indeed irreducible and rational. (Our calculation of the geometric genus above relied
on the assumption of $D$ being irreducible or at least connected). We do this by explicitly calculating a parametrization $\P^1 \to \barD_i$. 
\item The double cover induced by $N$ is non trivial on the Cayley cubic and $X_6$. We do this using the next Lemma \ref{lCoverHC}.
\end{enumerate}

This shows that we can apply Theorem \ref{tAuelPirutkaAlgebraicVersion} in this situation.

A Macaulay2 program doing the above calculations can be found at \cite{ABBP16}.
\end{proof}

\begin{lemma}\label{lCoverHC}
Let $\pi\colon Y \to B$ be a conic bundle defined over $k_0=\F_p$. Let
$S$ be an irreducible surface in $B$, defined over $k_0$, over which
the fibers of $Y$ generically consists of two distinct lines.  Let
$\tilde{S}\to S$ be the natural double cover of $S$ induced by
$\pi$. Then $\tilde{S}$ is irreducible if the following hold: there
exist two $k_0$-rational points $p_1, p_2\in S$ such that the fiber of
$Y$ over $p_1$ splits into two lines defined over $k_0$ whereas the
fiber over $p_2$ is irreducible over $k_0$ (and splits in a quadratic
extension of $k_0$ only).
\end{lemma}
\begin{proof}
Under the assumptions the double cover $\tilde{S} \to S$ is defined
over $k_0$. Suppose, by contradiction, that $\tilde{S}$ were
(geometrically) reducible. Then the Frobenius morphism $F$ would
either fix each irreducible component of $\tilde{S}$ as a set, or
interchange the two irreducible components. But since $S$ is defined
over $k_0$, this would mean that $F$ either fixes each of the two
lines as a set in every fiber over a $k_0$-rational point of the base,
or $F$ interchanges the two lines in every fiber over a $k_0$-rational
point. This contradicts the existence of $p_1, p_2$.
\end{proof}


\section{Desingularization of conic bundle fourfolds}\label{sSingularities}


The conic bundles considered above are singular. In this section, we
prove a criterion for the existence of a universally $\CH_0$-trivial
desingularization for such conic bundles.  Let $k$ be an algebraically
closed field of characteristic not 2. First recall the following
notion from \cite{A-CT-P} and \cite{CT-P16}.

\begin{definition}\label{dChow0UnivTriv}
A projective variety $X$ over a field $k$ has universally trivial
$\mathrm{CH}_0$ if for any extension $L \supset k$, the degree
homomorphism $\deg \colon \mathrm{CH}_0 (X_L)\to \Z$ is an
isomorphism.  A morphism $f \colon \tilde{Y} \to Y$ of projective
varieties over $k$ is called universally $\CH_0$-trivial if for any
overfield $L \supset k$, the pushforward $f_{\ast } \colon
\mathrm{CH}_0 (\tilde{Y}_L ) \to \mathrm{CH}_0 (Y_L)$ is an
isomorphism.
\end{definition}

We will make use of the following criterion to check that a resolution
of singularities is universally $\CH_0$-trivial.

\begin{proposition}\label{pResChowTriv}
A projective morphism $f \colon \tilde{Y} \to Y$ of projective
varieties over $k$ is universally $\CH_0$-trivial if for any
scheme-theoretic point $\xi$ of $Y$, the fiber $\tilde{Y}_{\xi}$, as a
scheme over the residue field $\kappa (\xi )$, is a projective variety
over $\kappa (\xi )$ with universally trivial $\mathrm{CH}_0$.
\end{proposition}

This is \cite[Prop.~1.8]{CT-P16}. Moreover, we use this in combination
with the following result, cf.\ \cite[Ex.~2]{HPT16}.

\begin{proposition}\label{pCHowTrivVar}
A projective, possibly reducible, geometrically connected variety $X =
\bigcup X_i$ over a field $k$ has universally trivial $\mathrm{CH}_0$
if each $X_i$ is geometrically irreducible, $k$-rational with isolated
singularities, and each intersection $X_i \cap X_j$ is either empty or
has a zero cycle of degree $1$.
\end{proposition}

Now we are ready to state our main result about the existence of
universally $\CH_0$-desingularizations of conic bundle fourfolds.  If
$Y \to B$ is a conic bundle, we colloquially say that $Y$ has a given
rank over a point of $B$ to mean that the fibral conic has that rank
at the respective point.

\begin{theorem} \label{tCH0trivial} Let $Y \to \P^3$ be a conic bundle
with reducible discriminant $X=X' \cup X''$. Let $D = X' \cap X''$ be
the intersection curve.  Assume:
\begin{itemize}
\item $X'$ and $X''$ are smooth along $D$.
\item $X'$ and $X''$ have only isolated nodes as singularities.
\item The rank of $Y$ at all nodes of $X'$ and $X''$ is $1$. 
\item $D = D_1 \cup \dots \cup D_n$ with $D_i$ irreducible reduced.
\item $D$ has only nodes as singularities.
\item The rank of $Y$ along $D$ is $2$ outside of the nodes of $D$.
\item The rank of $Y$ is $1$ on each node of the irreducible
components $D_i$ of $D$ (but not necessarily on the intersection
points between two irreducible components $D_i$ and $D_j$ of $D$).
\end{itemize}
Then $Y$ has a universally $CH_0$-trivial desingularization. 
\end{theorem}

\begin{remark}\label{rAdd}
Notice that both the Hassett--Pirutka--Tschinkel example from
\cite{HPT16} (see Section~\ref{eHPT}) and our new example (see
Proposition~\ref{pConicBundle}) satisfy these conditions. See
\cite{ABBP16} for computational details for our new example.
\end{remark}

\begin{theorem}\label{tMainApplication}
A very general conic bundle $Y \to \P^3$ over $\C$, defined by a homogeneous $3 \times 3$ matrix with entries of degrees
\[
	\begin{pmatrix}
	    7 & 4 & 4 \\
	    4 & 1 & 1 \\
	    4 & 1 & 1 
	 \end{pmatrix}
\]
is not stably rational.
\end{theorem}

\begin{proof}
Follows from Proposition~\ref{pConicBundle},
Theorem~\ref{tCH0trivial}, Remark \ref{rAdd} and the specialization
principle in unequal characteristic \cite[Thm.~1.12]{CT-P16}, as
employed in the proof of \cite[Thm.~1.20]{CT-P16}.
\end{proof}

To prove the above theorem some local computations are unavoidable.

\begin{proposition} \label{pNormalForms} 
Let $Y \to \P^3$ be a conic bundle with reducible discriminant $X=X'
\cup X''$. Let $D = X' \cap X''$ be the intersection curve and let
$X'$ and $X''$ be smooth along $D$. Let $D$ be reduced. 
Assume furthermore that the conic bundle has rank $2$ over the smooth locus of $D$. 
Finally let $P \in
D$ be a point.  Then we have the following local analytic normal forms

\begin{center}
\begin{tabular}{|c|c|c|c|c|c|c|c|c|c|c|}
\hline Geometry of $D$ at $P$ & Rank of $Y$ at $P$ & Normal form \\
\hline
\hline
smooth & $2$ &  $x^2 + sty^2 - z^2 = 0$ \\
\hline
node & $2$ &  $x^2 + sqy^2 - z^2 = 0$ \\
\hline 
node & $1$ & $x^2 + 2syz + (ty+uz)^2 = 0$ \\
\hline
\end{tabular}
\end{center}
Here $q = s + tu$ is quadratic in the completion $A=k\llbracket s,t,u
\rrbracket$ of the local ring at $P$ and $(x:y:z)$ are homogeneous
coordinates for $\P^2_A$.
\end{proposition}

\begin{proof}
Let $M$ be a $3 \times 3$ matrix over $A$ representing $Y$ locally
analytically around $P$.

First assume that $Y$ has rank $2$ at $P$.  Then $M$ has rank $2$ at
$P$ and we can, after a coordinate change on $\P^2_A$, assume that
\[
	M_P = \begin{pmatrix} 1 & 0 & 0 \\ 0 & 0 & 0 \\ 0 & 0 & -1 \end{pmatrix}
\]	
Therefore, the first $2$ diagonal entries are units in $A$ and we can,
after a further coordinate change, assume that
\[
	M = \begin{pmatrix} 1 & 0 & 0 \\ 0 & d & 0 \\ 0 & 0 & -1 \end{pmatrix}
\]	
with $d$ in $A$ a local equation for the discriminant of $Y$.

\medskip

\fbox{Case 1.} In the first case of the proposition, $D$ is smooth at
$P$ and therefore $X'$ and $X''$ intersect transversally around
$P$. Consequently, we can change coordinates in $A$ to obtain $X' =
\{s=0\}$ and $X'' = \{t=0\}$ with $s,t$ linear forms, i.e., $d=st$.
This gives the first normal form.

\medskip

\fbox{Case 2.} In the second case, $D$ has a node at $P$ and therefore
$X'$ and $X''$ are tangent at $P$. Let $X' = \{s=0\}$ and $X''
=\{q=0\}$. Since $X'$ is smooth at $P$, we can assume $s$ to be
linear.  Since $D = \{s=q=0\}$ has a node in $P$, it has two smooth
normal crossing branches there. We choose $t$ and $u$ to be local
linear equations of these branches on $\{s=0\}$.  Then
\[
	q = tu \mod s
\]
and we can write
\[
	q = \alpha s + tu.
\]
Now since $X''$ is smooth at $P$, we see that $\alpha$ must
be a unit. Absorbing $\alpha$ into $s$ we obtain $d=s(s+tu)$, which 
gives the second normal form.

\medskip

\fbox{Case 3.} In the third case, $Y$ has rank $1$ at $P$. By evaluating $M$ at $P$
and changing coordinates on $\P^2_A$ as above we can assume
\[
	M = \begin{pmatrix} 1 & 0   \\ 0 & N \end{pmatrix}
\]	
with $N$ a symmetric $2 \times 2$ matrix with entries in the maximal ideal of $A$.

Since $D$ has a node at $P$ we can, as before, assume that the
discriminant $\det N = -sq$ with $q=s+2tu$ and $s,t,u$ linear as
above. (The minus sign and the $2$ will be convenient later on).

Now $M$ has rank $2$ on $\{s=0\}$ outside the origin, and rank $1$ in
the origin.

In other words, $N$ is a matrix, defined locally around the origin in
the $(t,u)$-plane, and has rank $1$ everywhere in that plane except at
the origin, where it has rank $0$ (i.e., vanishes). Let
\begin{gather}\label{fN}
N = \begin{pmatrix}
\alpha & \beta \\
\beta & \gamma
\end{pmatrix}
\end{gather}
so that $\alpha (t,u) y^2 + 2 \beta (t,u ) yz + \gamma (t, u) z^2$ is
the associated quadratic form. Hence we must have
\begin{gather}\label{fIdentity}
\alpha \gamma - \beta^2 \equiv 0
\end{gather}
identically.  Now consider the prime factorizations of $\alpha, \beta
, \gamma$: if some prime $\pi$ divides $\alpha$ to odd order, it must
divide $\gamma$ to odd order, too, since it divides the square
$\beta^2$ to even order. Hence, in that case, $\pi$ divides all three
of them, which contradicts our assumption that the rank of $N$ does
not drop to $0$ on an entire curve germ through the origin in the
$(t,u)$-plane. Hence, $\alpha, \gamma$ are coprime squares, and we can
write
\[
	(y\; z) N (y\; z)^t \equiv (t' y+ u' z)^2 \mod s
	\iff N \equiv \begin{pmatrix} t'^2 & t'u' \\ t'u' & u'^2 \end{pmatrix} \mod s
\]
with $t',u'$ at least of degree $1$, since both vanish at $P$, and
coprime. It follows that we can write $N$ as
\[
	N = \begin{pmatrix} sf & sg \\ sg & sh \end{pmatrix}+ \begin{pmatrix} t'^2 & t'u' \\ t'u' & u'^2 \end{pmatrix} .
\]
Then the discriminant of $Y$ is
\[
	\det M = \det N = s \bigl( s(fh-g^2)+ (ht'^2 +2gt'u'+fu'^2)\bigr)
\]
Since this is equal to $-s^2 - 2stu$, and $t', u'$ are power series of
degree at least $1$ in $u,t$, comparing coefficients yields
$fh-g^2=-1$. We can therefore, after changing the fiber coordinates
$y$ and $z$, assume that
\[
	\begin{pmatrix} f & g \\ g & h \end{pmatrix} = \begin{pmatrix} 0 & 1 \\ 1 & 0 \end{pmatrix},
\]
The same coordinate change applied to $(t'y+u'z)$ gives
$(t''y+u''z)$. We obtain
\[
	\det M = s \bigl( -s + 2 t''u'' )
\]
Comparing coefficients with $\det M = -sq$ above we see that we can
take $\alpha = -1$, $t''=t$ and $u''=u$.  Then
\[
	M = \begin{pmatrix} 
	1 & 0 & 0  \\ 0 & t^2  & s+tu \\ 0 & s+tu & u^2 \end{pmatrix},	\quad
	\det M = -s(s+2tu)
\]
and we get the claimed normal form.
\end{proof}

Now we desingularize in these local coordinates.

\begin{proposition} \label{pNormalDesing} Let $Y \to \P^3$ be a conic
bundle with reducible discriminant $X=X' \cup X''$. Let $D = X' \cap
X''$ be the intersection curve and let $X'$ and $X''$ be smooth along
$D$. Assume furthermore that the conic bundle has rank $2$ over the
smooth locus of $D$. Finally let $P \in D$ be a point.  With the
normal forms from Proposition~\ref{pNormalForms} we have

\begin{center}
\begin{tabular}{|c|c|c|c|c|c|c|c|c|c|c|}
\hline Geometry of $D$ at $P$ & rank of $Y$ at $P$ & Singular Locus & Desingularization\\
\hline
\hline
smooth & $2$ & a line & blow up line \\
\hline
node & $2$ &  $2$ intersecting lines & blow up lines in \\
&&& arbitrary order (but \\
&&& not at the same time) \\ 
\hline 
node & $1$ & $2$ disjoint lines &  blow up lines in \\
&&& arbitrary order or at \\
&&& the same time. \\ 
\hline
\end{tabular}
\end{center}

$\quad$

In all three cases we have the following geometry. Consider the points
$P \in D$ where $Y$ has rank $2$. The fiber $Y_P$ over $P$ consist of
two lines which intersect in a point $P' \in Y_P$. Let $D' \subset Y$
be the closure of the locus of all such intersection points $P'$.
Then $D'$ is the singular locus of $Y$. Furthermore the covering $D'
\to D$ is $1:1$ over smooth points of $D$ and $2:1$ over rank $1$
nodes of $D$. Over rank $2$ nodes of $D$, $D'$ also has a node.
\end{proposition}

\begin{proof}
These are all straightforward calculations. See \cite{ABBP16} for a
Macaulay2 script doing them.
\end{proof}

\begin{remark}
Blowing up the intersection point of the two lines in the case of a
rank $2$ node does not improve things. While the strict transforms of
the two singular lines are separated we obtain a new singular line in
the exceptional divisor passing through both of the strict transforms.
\end{remark}

\begin{lemma}\label{lNode}
Let $\pi \colon Y \to \P^3$ be a conic bundle with discriminant $X$ a
surface having a node at $P\in X$. Assume $Y$ has rank $1$ at $P$ and
has rank $2$ on $X\smallsetminus \{P\}$ locally around $P$.  Then $Y$
is smooth over $P$ and has a local analytic normal form
\[
x^2 + sy^2 + 2tyz + uz^2=0
\]
where $(x:y:z)$ are homogeneous coordinates on $\P^2_A$ with
$A=k\llbracket s,t,u \rrbracket$.
\end{lemma}

\begin{proof}
Let $M$ be a $3 \times 3$ matrix over $A$ representing $Y$ locally
analytically around $P$. By evaluating $M$ at $P$ and changing
coordinates on $\P^2_A$ as above we can assume
\[
	M = \begin{pmatrix} 1 & 0   \\ 0 & N \end{pmatrix}
\]	
with $N$ a symmetric $2 \times 2$ matrix with entries in the maximal
ideal of $A$. Let
\[
N = \begin{pmatrix}
a & b\\
b & c
\end{pmatrix}.
\]
The Lemma follows if we can show that $a=b=c=0$ defines $P$ as a
reduced point because then we can choose $a,b,c$ as local
coordinates. Since $P$ is assumed to be a node $\det (N)=0$, we have
that the Jacobian ideal $J$ of $\det (N)$ defines $P$ as a reduced
point. Since $J \subset (a,b,c)$ by the product rule for derivatives,
our claim follows.  The fact that the total space of $Y$ is smooth
above $P$ is then a direct calculation.
\end{proof}

\begin{proof}[Proof of Theorem~\ref{tCH0trivial}]
We have to verify the hypotheses of Propositions~\ref{pResChowTriv}
and~\ref{pCHowTrivVar} for the resolutions $\tilde{Y}\to Y$ that we
produced in Proposition \ref{pNormalDesing}.

Since the singular locus of $X'$ and $X''$ consists only of isolated
nodes at rank $1$ points outside of $D$, the conic bundle $Y$ is
smooth outside of the preimage of $D$ by Lemma \ref{lNode}.

Let $D'$ be the closure of the locus of intersection points of lines
in fibers over $D$. By our assumptions in Theorem~\ref{tCH0trivial}
the conditions of Propositions \ref{pNormalForms} and
\ref{pNormalDesing} are satisfied. Furthermore, the local normal forms
studied in these Propositions are the only ones that occur.  It
follows that the singular locus of $Y$ is $D'$. Let $D'=D'_1 + \dots +
D'_n$ be its decomposition into irreducible components. By
Proposition~\ref{pNormalDesing} these components are birational to the
components of $D$.

We want to blow up the $D_i'$ in arbitrary order to obtain a
desingularization.  According to Proposition \ref{pNormalDesing}, the
only problem with our plan of blowing up the $D'_i$ in arbitrary order
is that over a rank $2$ node of $D_i$ both branches of $D'$ could get
blown up at the same time if this node is on only one irreducible
component $D_{i_0}'$. This would not lead to a desingularization over
rank $2$ nodes.  With our assumption that $Y$ has rank $1$ over all
nodes of irreducible components of $D$ we avoid this problem and
obtain a smoothing $\widetilde{Y}$ of $Y$.

It remains to describe the geometry of the fibers of $\sigma \colon
\widetilde{Y} \to Y$. We start by looking at fibers over closed
points. For this we consider the normal forms of
Proposition~\ref{pNormalForms}:

\fbox{Case 1.}  The normal form of $Y$ 
around a smooth point of $D$ is
\[
	x^2 + sty^2 - z^2 = 0.
\]
In these local coordinates $D'= \{s=t=x=z=0\}$. The Hessian matrix of
second derivatives of this normal form is
\[
\begin{pmatrix}
0&
      y^{2}&
      0&
      0&
      2 t y&
      0\\
      y^{2}&
      0&
      0&
      0&
      2 s y&
      0\\
      0&
      0&
      0&
      0&
      0&
      0\\
      0&
      0&
      0&
      2&
      0&
      0\\
      2 t y&
      2 s y&
      0&
      0&
      2 s t&
      0\\
      0&
      0&
      0&
      0&
      0&
      {-2}\\
      \end{pmatrix}.
\]
At $(0:0:0:0:1:0) \in D'$ this matrix has rank $4$. Therefore the
fiber of $\sigma$ over this point is a $\P^1 \times \P^1$.

\fbox{Case 2.}  The normal form of $Y$ 
around a singular rank $2$ point of $D$ is
\[
	x^2 + s(s+tu)y^2 - z^2 = 0.
\]
The curve $D'$ consists of two lines that intersect in the point 
$$y = (0:0:0:0:1:0) \in D'.$$
We  blow up in two steps
\[
	\widetilde{Y} \xrightarrow{\sigma_2} Y' \xrightarrow{\sigma_1} Y
\]
with $\sigma_1$ blowing up one of the lines and $\sigma_2$ blowing up
the strict transform of the other line.

The Hessian matrix of the above normal form has rank $3$ in
$y$. Therefore the fiber of $\sigma_1$ over $y$ is a quadric cone
$C$. Now, the strict transform of the other line intersects this
quadric cone in one point. After a coordinate change, $Y'$ has the
same normal form as Case 1 above. Therefore the Hessian matrix at the
intersection point $y'$ of $C$ with the strict transform of the second
line has rank $4$. So the fiber of $\sigma_2$ over $y'$ is a $\P^1
\times \P^1$. The fiber of $\sigma = \sigma_2 \circ \sigma_1$ over $y$
consists then of the strict transform of the quadric cone $C$ under
$\sigma_2$ and a $\P^1 \times \P^1$.

\fbox{Case 3.}  The normal form of $Y$ 
around a singular rank $1$ point of $D$ is
\[
	x^2 +2syz+(ty+uz)^2 =0
\]
The curve $D'$ consists again of two lines, but this time these lines do not intersect.
Over the singular point of $D$ we have therefore $2$ points on $D'$, namely 
\[
 y = (0:0:0:0:0:1) \quad \text{and} \quad y' = (0:0:0:0:1:0).
\]
The Hessian matrix of the normal form above has rank $4$ in each point and therefore
the fiber of $\sigma$ is $\P^1 \times \P^1$ in both cases.

It remains now to consider the fibers over components of $D'$. By the
above calculations the fibers over smooth points of $D'$ are
isomorphic to $\P^1 \times \P^1$. The fibers over each curve component
of $D'$ are therefore birational to $\P^1 \times \P^1$-bundles. By
Tsen's theorem, these $\P^1 \times \P^1$ bundles are Zariski locally
trivial over the components $D_i'$, so we conclude using
Proposition~\ref{pResChowTriv}.
\end{proof}

\end{document}